\documentclass[11pt,a4paper]{amsart}
\usepackage{amssymb, amstext, amscd, amsmath, color}

\usepackage{url}

\usepackage{tikz}
\usetikzlibrary{intersections,arrows,decorations.markings}
\tikzset{->-/.style={decoration={
  markings,
  mark=at position #1 with {\arrow{>}}},postaction={decorate}}}

\usepackage{amsfonts}
\usepackage{lscape}
\usepackage{amsthm}
\usepackage{enumerate}
\usepackage{color}
\usepackage{amsfonts}
\usepackage{array}

\usepackage{epsfig}
\usepackage{eucal}
\usepackage{latexsym}
\usepackage{mathrsfs}
\usepackage{textcomp}
\usepackage{verbatim}

\newtheorem{thm}{Theorem}[section]

\newtheorem{claim}[thm]{Claim}

\newtheorem{cor}[thm]{Corollary}

\newtheorem{lem}[thm]{Lemma}

\newtheorem*{thm1.1}{Theorem 1.1}
\newtheorem*{thm1.2}{Theorem 1.2}
\newtheorem*{thm1.3}{Theorem 1.3}
\numberwithin{equation}{section}

  \newcommand{\A}{{\mathcal{A}}}

  \newcommand{\E}{{\mathcal{E}}}
  \newcommand{\F}{{\mathcal{F}}}
  \newcommand{\G}{{\mathcal{G}}}

  \newcommand{\W}{{\mathcal{W}}}
  \newcommand{\X}{{\mathcal{X}}}

\begin{document}

\title{A constructive characterisation of circuits in the simple $(2,1)$-sparse matroid}
\author[T. McCourt]{T.A. McCourt}
\address{School of Computing, Electronics and Mathematics\\ Plymouth University\\ Plymouth\\ PL4 8AA\\ U.K. }
\email{tom.a.mccourt@gmail.com}
\author[A. Nixon]{A. Nixon}
\address{Department of Mathematics and Statistics\\ Lancaster University\\ Lancaster\\
LA1 4YF \\ U.K. }
\email{a.nixon@lancaster.ac.uk}
\begin{abstract}
A simple graph $G=(V,E)$ is a $(2,1)$-circuit if $|E|=2|V|$ and $|E(H)|\leq 2|V(H)|-1$ for every proper subgraph $H$ of $G$. Motivated, in part, by ongoing work to understand unique realisations of graphs on surfaces, we derive a constructive characterisation of $(2,1)$-circuits. The characterisation uses the well known 1-extension and $X$-replacement operations as well as several summation moves to glue together $(2,1)$-circuits over small cutsets.
\end{abstract}

\maketitle

\section{Introduction}

For a finite (multi)graph $G=(V,E)$ let $i_G(X)$ denote the number of edges induced by $X\subseteq V$. The graph $G$ is \emph{$(k,\ell)$-tight} if $|E|=k|V|-\ell$ and $i_G(X)\leq k|V|-\ell$ for all $X\subseteq V$. For $\ell <2k$ the (edge sets of) $(k,\ell)$-tight graphs form the bases of a matroid \cite{L&S,Whi3}. These matroids are examples of count matroids (see \cite{Fra}); we refer to them as the $(k,\ell)$-sparse matroid, and when loops and multiple edges are prohibited, as the simple $(k,\ell)$-sparse matroid.

A simple (respectively multi-)graph $G$ is a \emph{$(k,\ell)$-circuit} if $G$ is the graph induced by a circuit in the simple $(k,\ell)$-sparse matroid (respectively the $(k,\ell)$-sparse matroid). Equivalently a simple (respectively multi-)graph $G$ is a $(k,\ell)$-circuit if $|E|=k|V|-\ell+1$ and $i_G(X)\leq k|V|-\ell$ for all $X\subsetneq V$.

A constructive characterisation of a class of graphs is a method of building all graphs in the class from certain base graphs by elementary operations. 
Constructive characterisations of classes of graphs (or other objects) from small base cases by elementary local transformations are natural and relevant to a variety of problems. In particular in the field of combinatorial optimization \cite{K&V}. As particular motivation for us we seek a means to understand when a generic realisation of a graph on a surface is globally rigid (unique up to isometries), see \cite{JMN,JNstress} for details on the geometry of this problem. A related construction for circuits in the $(2,3)$-sparse matroid \cite{B&J} was a vital aspect of the characterisation of global rigidity in the plane \cite{J&J} and we expect that the construction here, along with the construction for circuits in the simple $(2,2)$-sparse matroid, will be crucial in establishing analogues for global rigidity on surfaces supporting either two (e.g. the cylinder) or one (e.g. the cone or torus) tangentially acting isometries. See \cite[Conjecture $5.7$]{Nix} and \cite[Conjecture $9.1$]{JMN}. 

Constructive characterisations of $(k,\ell)$-tight (multi)graphs are known when $k=\ell$ \cite{Tay} and more generally when $l\leq k$ \cite{F&S}. 
In the case of simple graphs each intermediate graph in the recursive construction needs to be simple; this necessitates establishing new constructions.
When $k=2$ constructions are known for several classes \cite{N&O,NOP,NOP2}.

For $(k,\ell)$-circuits less is known, however some general constructions, such as when $k=\ell$ can be extracted from work on tree packings \cite{Fr&S}. In this paper we are interested in simple $(k,\ell)$-circuits. For $k=1$ the problem is easy. When $k=2$, and $\ell=3$, there is an elegant result of Berg and Jord\'{a}n \cite{B&J}. It is easy to check that the minimum degree, $\delta(G)$, in a $(2,3)$-circuit $G$ is equal to three. They proved first that every $(2,3)$-circuit which was 3-connected contains an \emph{admissible} degree three vertex, that is a degree three vertex which can be removed and an edge added between its neighbours in such a way that the resulting graph is a $(2,3)$-circuit. To complete their constructive characterisation they observed that whenever $G$ was not 3-connected there is a 2-separation $\{x,y\}$ and $xy\notin E$. Such a graph can be broken into two smaller $(2,3)$-circuits by separating over this cut and adding the edge $xy$ to both parts. (This is the inverse of the well known 2-sum operation.)

For $(2,2)$-circuits one must distinguish between the multigraph and simple graph cases. For multigraphs a constructive characterisation occurs as a special case of a characterisation of highly $k$-tree connected graphs \cite{Fr&S}. For simple graphs there is a constructive characterisation in \cite{Nix}. It is a nontrivial extension of Berg and Jord\'{a}n's result in the following senses: firstly since $K_4$ can occur as a subgraph of any $(2,2)$-circuit, it is non-trivial to preserve simplicity; secondly the connectivity level required to guarantee admissibility is higher (requiring 3-connectivity and essential 4-edge-connectivity); and thirdly when these connectivity assumptions do not hold there are several separation moves required.

In this paper we consider simple $(2,1)$-circuits and in particular, we prove a constructive characterisation of such circuits where every intermediate graph is also simple. 
 At a high level the scheme is analogous to the above cases: we establish the connectivity level required to guarantee an admissible vertex and when this fails we define separation moves to pull apart $(2,1)$-circuits into smaller $(2,1)$-circuits. 
However, in the $(2,1)$-circuit case there are numerous complications to previous cases.
 Firstly observe that $(2,1)$-circuits need not have $\delta(G)=3$, thus we are forced to consider a degree four operation. Secondly since $K_4$ is not a base in the simple $(2,1)$-sparse matroid we separate out two cases in our result guaranteeing an admissible vertex of degree three. 
Thirdly, due to the connectivity level required to guarantee an admissible node, there are six different separations that can occur.

\subsection{Terminology}

We now define the set $\G$ of base graphs for our constructive characterisation.
Note first that $K_5$ is the unique simple $(2,1)$-circuit on at most five vertices. There are five distinct graphs formed from deleting three edges from $K_6$. Deleting a one factor leaves a 4-regular graph which has an admissible vertex by Theorem \ref{thm:4reg}. Deleting a degree three star does not give a $(2,1)$-circuit. The other three possibilities are the complements of $G57,G59$ and $G60$ as listed in \cite{R&W} (see also Figure \ref{fig:3fromK6}). 

\begin{figure}
\begin{center}
\begin{tikzpicture}[scale=1, vertex/.style={circle,inner sep=2,fill=black,draw}, vertex2/.style={circle,inner sep=4,fill=black,draw}]

\coordinate (v1) at (0,0);
\coordinate (v2) at (1,0);
\coordinate (v3) at (2,0.5);
\coordinate (v4) at (0,1);
\coordinate (v5) at (1.25,1.25);
\coordinate (v6) at (0.5,2);

\node at (v1) [vertex]{};
\node at (v2) [vertex]{};
\node at (v3) [vertex]{};
\node at (v4) [vertex]{};
\node at (v5) [vertex]{};
\node at (v6) [vertex]{};

\draw (v1) -- (v2);
\draw (v1) -- (v5);
\draw (v1) -- (v4);
\draw (v2) -- (v3);
\draw (v2) -- (v4);
\draw (v2) -- (v5);
\draw (v3) -- (v5);
\draw (v4) -- (v5);
\draw (v4) -- (v6);
\draw (v5) -- (v6);
\draw (v2) -- (v6);
\draw (v3) -- (v4);

\node at (1,0) [label=south:$\overline{G57}$]{};

\coordinate (v7) at (3.5,0);
\coordinate (v8) at (4.5,0);
\coordinate (v9) at (3.5,1);
\coordinate (v10) at (4.5,1);
\coordinate (v11) at (3.5,2);
\coordinate (v12) at (4.5,2);

\node at (v7) [vertex]{};
\node at (v8) [vertex]{};
\node at (v9) [vertex]{};
\node at (v10) [vertex]{};
\node at (v11) [vertex]{};
\node at (v12) [vertex]{};

\draw (v7) -- (v11);
\draw (v8) -- (v12);
\draw (v7) -- (v8);
\draw (v9) -- (v10);
\draw (v11) -- (v12);
\draw (v11) -- (v10);
\draw (v9) -- (v8);
\draw (v7) -- (v10);
\draw (v9) -- (v12);

\draw[bend left=30] (v7) edge (v11);

\node at (4,0) [label=south:$\overline{G59}$]{};

\coordinate (v13) at (6,0);
\coordinate (v14) at (7,0);
\coordinate (v15) at (6,1);
\coordinate (v16) at (7,1);
\coordinate (v17) at (6.5,1.5);
\coordinate (v18) at (6.5,2);

\node at (v13) [vertex]{};
\node at (v14) [vertex]{};
\node at (v15) [vertex]{};
\node at (v16) [vertex]{};
\node at (v17) [vertex]{};
\node at (v18) [vertex]{};

\draw (v13) -- (v14);
\draw (v13) -- (v15);
\draw (v13) -- (v16);
\draw (v14) -- (v15);
\draw (v14) -- (v16);
\draw (v15) -- (v17);
\draw (v15) -- (v18);
\draw (v16) -- (v17);
\draw (v16) -- (v18);
\draw (v18) -- (v17);

\draw[bend left=50] (v13) edge (v18);
\draw[bend right=50] (v14) edge (v18);

\node at (6.5,0) [label=south:$\overline{G60}$]{};

\end{tikzpicture}
\end{center}
\caption{$\overline{G57}$ and $\overline{G59}$ are base graphs on six vertices. $\overline{G60}$ can be reduced to $K_5$ by a 1-reduction.}
\label{fig:3fromK6}
\end{figure}
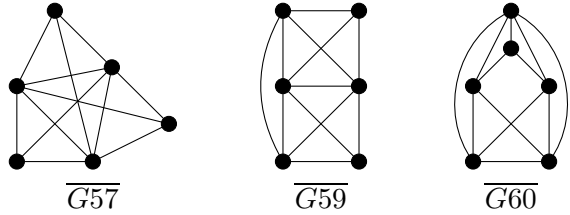

Examining the list of small graphs in \cite{R&W}, there are 65 non-isomorphic graphs formed from deleting seven edges from $K_7$; of these 34 are $(2,1)$-circuits.
Two of these are 4-regular and hence contain an admissible node by Theorem \ref{thm:4reg}.
Of the remainder note that if there is a node not in a copy of $K_4$ then either adding a missing edge creates a $K_5$ subgraph or that node is admissible. This observation allows us to easily spot that 29 of the remaining $(2,1)$-circuits contain an admissible node. The remaining three do not contain admissible nodes; they are the complements of $G293,G308$ and $G312$ as listed in \cite{R&W} (see also Figure \ref{fig:7fromK7}).

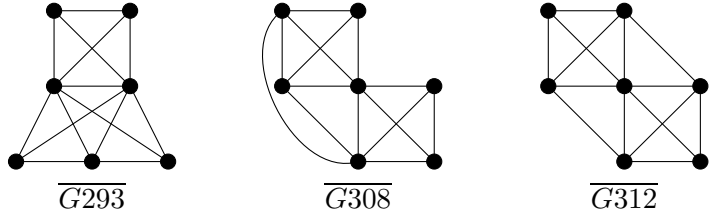
\begin{figure}
\begin{center}
\begin{tikzpicture}[scale=1, vertex/.style={circle,inner sep=2,fill=black,draw}, vertex2/.style={circle,inner sep=4,fill=black,draw}]

\coordinate (v1) at (0,0);
\coordinate (v2) at (1,0);
\coordinate (v3) at (2,0);
\coordinate (v4) at (0.5,1);
\coordinate (v5) at (1.5,1);
\coordinate (v6) at (0.5,2);
\coordinate (v7) at (1.5,2);

\node at (v1) [vertex]{};
\node at (v2) [vertex]{};
\node at (v3) [vertex]{};
\node at (v4) [vertex]{};
\node at (v5) [vertex]{};
\node at (v6) [vertex]{};
\node at (v7) [vertex]{};

\draw (v1) -- (v2);
\draw (v1) -- (v4);
\draw (v1) -- (v5);
\draw (v2) -- (v3);
\draw (v2) -- (v4);
\draw (v2) -- (v5);
\draw (v3) -- (v4);
\draw (v3) -- (v5);
\draw (v4) -- (v5);
\draw (v4) -- (v6);
\draw (v4) -- (v7);
\draw (v5) -- (v7);
\draw (v5) -- (v6);
\draw (v6) -- (v7);

\node at (1,0) [label=south:$\overline{G293}$]{};

\coordinate (v8) at (4.5,0);
\coordinate (v9) at (5.5,0);
\coordinate (v10) at (3.5,1);
\coordinate (v11) at (4.5,1);
\coordinate (v12) at (5.5,1);
\coordinate (v13) at (3.5,2);
\coordinate (v14) at (4.5,2);

\node at (v8) [vertex]{};
\node at (v9) [vertex]{};
\node at (v10) [vertex]{};
\node at (v11) [vertex]{};
\node at (v12) [vertex]{};
\node at (v13) [vertex]{};
\node at (v14) [vertex]{};

\draw (v8) -- (v9);
\draw (v8) -- (v10);
\draw (v8) -- (v11);
\draw (v8) -- (v12);
\draw (v9) -- (v11);
\draw (v9) -- (v12);
\draw (v10) -- (v11);
\draw (v10) -- (v13);
\draw (v10) -- (v14);
\draw (v11) -- (v12);
\draw (v11) -- (v13);
\draw (v11) -- (v14);
\draw (v13) -- (v14);

\draw[bend left=80] (v8) edge (v13);

\node at (4.5,0) [label=south:$\overline{G308}$]{};

\coordinate (v15) at (8,0);
\coordinate (v16) at (9,0);
\coordinate (v17) at (8,1);
\coordinate (v18) at (9,1);
\coordinate (v19) at (7,1);
\coordinate (v20) at (7,2);
\coordinate (v21) at (8,2);

\node at (v15) [vertex]{};
\node at (v16) [vertex]{};
\node at (v17) [vertex]{};
\node at (v18) [vertex]{};
\node at (v19) [vertex]{};
\node at (v20) [vertex]{};
\node at (v21) [vertex]{};

\draw (v20) -- (v16);
\draw (v15) -- (v19);
\draw (v19) -- (v21);
\draw (v19) -- (v18);
\draw (v18) -- (v16);
\draw (v21) -- (v15);
\draw (v21) -- (v20);
\draw (v20) -- (v19);
\draw (v15) -- (v16);
\draw (v18) -- (v21);
\draw (v18) -- (v15);

\node at (8,0) [label=south:$\overline{G312}$]{};

\end{tikzpicture}
\end{center}
\caption{Base graphs on seven vertices.}
\label{fig:7fromK7}
\end{figure}

A further five base graphs of orders eight and nine arise in the proof of the recursive construction; there are four on eight vertices (graphs $S_1$, $S_2$, $S_3$ and $S_4$ in Figure \ref{fig:S234}) and there is one on nine vertices (graph $S_5$ in Figure \ref{fig:base9}).

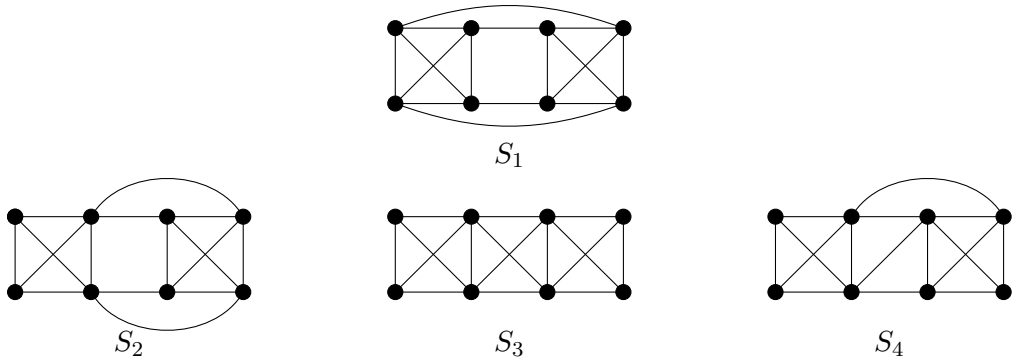
\begin{figure}
\begin{center}
\begin{tikzpicture}[scale=1, vertex/.style={circle,inner sep=2,fill=black,draw}, vertex2/.style={circle,inner sep=4,fill=black,draw}]

\coordinate (v1) at (0,0);
\coordinate (v2) at (1,0);
\coordinate (v3) at (2,0);
\coordinate (v4) at (3,0);
\coordinate (v5) at (0,1);
\coordinate (v6) at (1,1);
\coordinate (v7) at (2,1);
\coordinate (v8) at (3,1);

\node at (v1) [vertex]{};
\node at (v2) [vertex]{};
\node at (v3) [vertex]{};
\node at (v4) [vertex]{};
\node at (v5) [vertex]{};
\node at (v6) [vertex]{};
\node at (v7) [vertex]{};
\node at (v8) [vertex]{};

\draw (v1) -- (v4);
\draw (v5) -- (v8);
\draw (v1) -- (v5);
\draw (v2) -- (v6);
\draw (v3) -- (v7);
\draw (v4) -- (v8);
\draw (v1) -- (v6);
\draw (v2) -- (v5);
\draw (v3) -- (v8);
\draw (v4) -- (v7);

\draw[bend right=60] (v2) edge (v4);

\draw[bend left=60] (v6) edge (v8);

\node at (1.5,-0.25)[label=south:$S_2$]{};


\coordinate (v9) at (5,0);
\coordinate (v10) at (6,0);
\coordinate (v11) at (7,0);
\coordinate (v12) at (8,0);
\coordinate (v13) at (5,1);
\coordinate (v14) at (6,1);
\coordinate (v15) at (7,1);
\coordinate (v16) at (8,1);

\node at (v9) [vertex]{};
\node at (v10) [vertex]{};
\node at (v11) [vertex]{};
\node at (v12) [vertex]{};
\node at (v13) [vertex]{};
\node at (v14) [vertex]{};
\node at (v15) [vertex]{};
\node at (v16) [vertex]{};

\draw (v9) -- (v12);
\draw (v13) -- (v16);
\draw (v9) -- (v13);
\draw (v10) -- (v14);
\draw (v11) -- (v15);
\draw (v12) -- (v16);
\draw (v9) -- (v14);
\draw (v10) -- (v13);
\draw (v11) -- (v16);
\draw (v12) -- (v15);

\draw (v11) -- (v14);
\draw (v10) -- (v15);

\node at (6.5,-0.25)[label=south:$S_3$]{};

\coordinate (v17) at (10,0);
\coordinate (v18) at (11,0);
\coordinate (v19) at (12,0);
\coordinate (v20) at (13,0);
\coordinate (v21) at (10,1);
\coordinate (v22) at (11,1);
\coordinate (v23) at (12,1);
\coordinate (v24) at (13,1);

\node at (v17) [vertex]{};
\node at (v18) [vertex]{};
\node at (v19) [vertex]{};
\node at (v20) [vertex]{};
\node at (v21) [vertex]{};
\node at (v22) [vertex]{};
\node at (v23) [vertex]{};
\node at (v24) [vertex]{};

\draw (v17) -- (v20);
\draw (v21) -- (v24);
\draw (v17) -- (v21);
\draw (v18) -- (v22);
\draw (v19) -- (v23);
\draw (v20) -- (v24);
\draw (v17) -- (v22);
\draw (v18) -- (v21);
\draw (v19) -- (v24);
\draw (v20) -- (v23);
\draw (v18) -- (v23);

\draw[bend left=60] (v22) edge (v24);

\node at (11.5,-0.25)[label=south:$S_4$]{};


\coordinate (v1) at (5,2.5);
\coordinate (v2) at (6,2.5);
\coordinate (v3) at (5,3.5);
\coordinate (v4) at (6,3.5);


\coordinate (v6) at (7,2.5);
\coordinate (v7) at (8,2.5);
\coordinate (v8) at (7,3.5);
\coordinate (v9) at (8,3.5);


\node at (v1) [vertex]{};
\node at (v2) [vertex]{};
\node at (v3) [vertex]{};
\node at (v4) [vertex]{};
\node at (v5) [vertex]{};
\node at (v6) [vertex]{};
\node at (v7) [vertex]{};
\node at (v8) [vertex]{};
\node at (v9) [vertex]{};

\draw (v1) -- (v7);
\draw (v1) -- (v3);
\draw (v1) -- (v4);
\draw (v2) -- (v3);
\draw (v2) -- (v4);
\draw (v3) -- (v9);
\draw (v6) -- (v8);
\draw (v6) -- (v9);
\draw (v7) -- (v8);
\draw (v7) -- (v9);

\draw[bend right=20] (v1) edge (v7);

\draw[bend left=20] (v3) edge (v9);

\node at (6.5,2.25)[label=south:$S_1$]{};

\end{tikzpicture}
\end{center}
\caption{Base graphs on eight vertices.}
\label{fig:S234}
\end{figure}

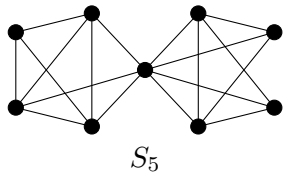
\begin{figure}
\begin{center}
\begin{tikzpicture}[scale=1, vertex/.style={circle,inner sep=2,fill=black,draw}, vertex2/.style={circle,inner sep=4,fill=black,draw}]

\coordinate (v1) at (0,1);
\coordinate (v2) at (0,0);
\coordinate (v3) at (1,-0.25);
\coordinate (v4) at (1,1.25);
\coordinate (v5) at (1.7,0.5);
\coordinate (v6) at (2.4,-0.25);
\coordinate (v7) at (2.4,1.25);
\coordinate (v8) at (3.4,1);
\coordinate (v9) at (3.4,0);

\node at (v1) [vertex]{};
\node at (v2) [vertex]{};
\node at (v3) [vertex]{};
\node at (v4) [vertex]{};
\node at (v5) [vertex]{};
\node at (v6) [vertex]{};
\node at (v7) [vertex]{};
\node at (v8) [vertex]{};
\node at (v9) [vertex]{};

\draw (v1) -- (v2);
\draw (v1) -- (v3);
\draw (v1) -- (v4);
\draw (v2) -- (v3);
\draw (v2) -- (v4);
\draw (v2) -- (v5);
\draw (v3) -- (v4);
\draw (v3) -- (v5);
\draw (v4) -- (v5);
\draw (v5) -- (v6);
\draw (v5) -- (v7);
\draw (v5) -- (v8);
\draw (v5) -- (v9);
\draw (v6) -- (v7);
\draw (v6) -- (v8);
\draw (v6) -- (v9);
\draw (v7) -- (v8);
\draw (v7) -- (v9);

\node at (1.7,-0.25)[label=south:$S_5$]{};

\end{tikzpicture}
\end{center}
\caption{Base graphs on 9 vertices.}
\label{fig:base9}
\end{figure}

Hence we define $\G$ to be the set $\{K_5,\overline{G57},\overline{G59},\overline{G293},\overline{G308},\overline{G312}, S_1, S_2,S_3,S_4,S_5\}.$

For a graph $G$ let $\delta(G)$ denote the minimum degree of a vertex in $G$. A $k$-edge-cutset is a set of $k$ edges whose removal disconnects a graph. If one of the components, of the resulting disconnected graph, contains only a single vertex then we say that the $k$-edge-cutset is \emph{trivial}. Otherwise it is \emph{non-trivial}. A graph is \emph{essentially $k$-edge-connected} if it has no non-trivial $(k-1)$-edge-cutsets. Since $(2,1)$-circuits have $\delta(G)\in \{3,4\}$ this will be a useful notion for us.

Let $G=(V,E)$ be a graph and let $G'$ be formed from $G$ by deleting an edge $xy$ from $E$ and adding a new vertex $v$ and edges $xv,yv,zv$ for $x,y,z\in V$. This operation is known as a 1-extension (elsewhere it is also referred to as the Henneberg 2 move \cite{NR,Tay,Whi3}). The inverse is known as a 1-reduction.
Also let $G''$ be formed from $G$ by deleting two non-adjacent edges $xy,zw$ from $E$ and adding a new vertex $v$ and edges $xv,yv,zv,wv$ for $x,y,z,w\in V$. This operation is known as $X$-replacement \cite{NR,Whi3}. 

As outlined above, when a circuit is sufficiently connected we will show that either a 1-reduction or an inverse $X$-replacement can be applied to form a new circuit. In the cases where a circuit is not sufficiently connected we will apply inverse `summation' moves to reduce to a smaller circuit. 

There will remain a handful of cases where the circuit, $G$ say, is invariant under these inverse `summation' moves. In order to deal with these cases we will construct a new circuit $G^*$ which is a, sufficiently connected, multigraph. We will then show that there exists a 1-reduction or inverse $X$-replacement that can be applied to $G^*$ to yield a new circuit. Finally we use this to show that an analogous 1-reduction or $X$-replacement could have been applied in $G$ to yield a new circuit.

Let $M^*(2,1)$ denote the set of all multigraphs that are $(2,1)$-circuits and let $M(2,1)$ denote the set of all simple graphs that are $(2,1)$-circuits.
The multigraph circuits that we construct in the manner referred to above are all elements of the subset $\mathcal{M}\subseteq M^*(2,1)$ defined as follows.
\begin{itemize}
\item[] Let $\mathcal{M}'$ be the subset of $M^*(2,1)$ where $G\in \mathcal{M}'$  if and only if 
\begin{itemize}
\item $G$ is either 3-connected or has fewer than four vertices;
\item the maximum edge multiplicity is three;
\item all the vertices in $G$ incident with a multiple edge have degree greater than three; 
\item if $G$ contains a loop, then the vertex incident with the loop has degree greater than three and it is not incident with any multiple edges; and
\item if $G=(V,E)$ contains a vertex, $v$ say, incident to two loops then $V=\{v\}$. 
\end{itemize}
\item[] Then $\mathcal{M}=M(2,1)\cup\mathcal{M}'$.
\end{itemize}

Note that $M(2,1)\subset\mathcal{M}\subset M^*(2,1)$. With this in mind, letting $G\in\mathcal{M}$, we call
\begin{itemize}
\item a vertex of degree three in $G$ a \emph{node};
\item a node $v$ of $G$ \emph{admissible} if there is a 1-reduction removing $v$ that results in a $(2,1)$-circuit in which the added edge is between previously non-adjacent vertices;
\item a degree four vertex $v$ in $G$ \emph{admissible} if there is an inverse $X$-replacement removing $v$ that results in a $(2,1)$-circuit in which the added edges are between pairs of previously non-adjacent vertices. 
\end{itemize}  
As $M(2,1)\subset\mathcal{M}$, in the case where $G\in M(2,1)$, the resulting circuit for an admissible 1-reduction should also be in $M(2,1)$. (The reader may find the motivation behind the results in Section \ref{sec:hen2} clearer if these results are thought of in the restricted case of $M(2,1)$; we will not need the general case until the end of the proof of Theorem \ref{thm:mainresult}.)
 We will use $\G^*$ to refer to the set of base graphs in $\mathcal{M}$. The elements of $\G^* \backslash M(2,1)$ will be derived in Section \ref{sec:hen2}.

\subsection{Results}

Our main result is as follows.

\begin{thm}\label{thm:mainresult}
A simple graph $G$ is a $(2,1)$-circuit if and only if it can be generated recursively from $G\in \G$ by applying 1-extensions and $X$-replacements sequentially within connected components and taking sums of connected components.
\end{thm}

In order to prove the theorem we first establish admissibility when the minimum degree is 4.

\begin{thm}\label{thm:4reg}
Every essentially 5-edge-connected $(2,1)$-circuit $G\in M(2,1)$ not equal to $K_5$ with $\delta(G)=4$ contains an admissible vertex.
\end{thm}

From there, the key technical step to proving Theorem \ref{thm:mainresult} is to establish admissibility, in the $\delta(G)=3$ case, when $G$ is sufficiently connected.

\begin{thm}\label{thm:Hen2admissible}
Let $G=(V,E)$ be a $3$-connected $(2,1)$-circuit with $\delta(G)=3$ and $G\not\in \G$. Suppose either
\begin{enumerate}[(i)]
\item $G$ is essentially 5-edge-connected and there is no proper critical set or
\item $G$ is essentially 4-edge-connected and there is a proper critical set.
\end{enumerate}
Then $G$ contains an admissible node.
\end{thm}

Proper critical sets are defined in Section \ref{sec:hen2} and the reason for splitting the above theorem into these two cases is purely technical.

\subsection{Outline of paper}

In Section \ref{sec:4reg} we prove Theorem \ref{thm:4reg} establishing that when $\delta(G)=3$ then either a $(2,1)$-circuit has an admissible vertex or it has a small separating set of edges.
In Section \ref{sec:hen2} we prove Theorem \ref{thm:Hen2admissible} establishing that whenever a $(2,1)$-circuit has $\delta(G)=4$ then either it has an admissible vertex or it has a small separating set of edges. In Section \ref{sec:sumsec} we deal with graphs that are $(2,1)$-circuits but fail to be sufficiently connected. In Section \ref{sec:construction} we combine our results to prove Theorem \ref{thm:mainresult}.


\section{$X$-replacement}
\label{sec:4reg}

An \emph{X-replacement} is the deletion of two non-adjacent edges $ab,cd$ and the addition of a vertex $v$ adjacent to $a,b,c,d$. The inverse operation is the deletion of a degree four vertex and the addition of two non-adjacent edges between the neighbours. 
For $G=(V,E)$, define $f(G)=2|V|-|E|$, so that a graph $G$ is a $(2,1)$-circuit if and only if $f(G)=0$ and every proper subgraph $H$ satisfies $f(H)\geq 1$.
If $G$ is a $(2,1)$-circuit and $v \in V$, then we say that $v$ is \emph{admissible} if there is an inverse X-replacement on $v$ resulting in a $(2,1)$-circuit.

\begin{lem}\label{lem:4reg}
A graph $G$ is a $(2,1)$-circuit with $\delta(G)=4$ if and only if $G$ is connected and 4-regular.
\end{lem}

\begin{proof}
Let $G=(V,E)$.
If $G$ is a $(2,1)$-circuit with $\delta(G)=4$, then $G$ is 4-regular and a simple counting argument
implies that $G$ is connected.
Conversely, if $G$ is connected and 4-regular then $|E|=2|V|$. Suppose $H=(V',E')$ is a proper subgraph of $G$ with $|E'| \geq 2|V'|$ then $H$ has average degree at least four. Since $H$ is proper and $G$ is connected we contradict the 4-regularity of $G$.
\end{proof}

With this lemma, Theorem \ref{thm:4reg} can be extracted from \cite{BJF}. For completeness we provide an alternate proof.


\begin{proof}[Proof of Theorem \ref{thm:4reg}]

Let $G=(V,E)$.

\begin{claim}\label{claim:simple1}
Let $v \in V$. Then $v$ is non-admissible if and only if every possible inverse X-replacement on $v$ results in a non-simple graph.
\end{claim}

\begin{proof}[Proof of Claim]
If every inverse X-replacement on $v$ creates a multigraph, then clearly $v$ is non-admissible.
For the converse let $G'$ be the result of the inverse X-replacement. By Lemma \ref{lem:4reg} we need to check that 4-regularity and connectedness are preserved. The first of these is clear and if $G'$ is not connected then $v$ is a cut-vertex, but since $v$ has degree four this contradicts the assumption that G is essentially 5-edge-connected.
\end{proof}

It follows that 
$G'$ is not simple if and only if in $G$, $v$ is (a) contained in a copy of $K_4$ or (b) there is $x \in N(v)$ (the neighbour set of $v$) adjacent to every other vertex in $N(v)$.

We next show that no subgraph of $G$ is isomorphic to $K_4$.
Suppose otherwise, i.e., that $H$ is a subgraph of $G$ isomorphic to $K_4$. As $G$ is 4-regular there is a 4-edge-cutset between $H$ and $G-H$, since $G\neq K_5$ this cutset is non-trivial.

We complete the proof by showing that either $v$ or one of its neighbours are admissible.
(In fact, either $v$ or three of its neighbours are admissible.)
Let $N(v)={w,x,y,z}$ and note that $v$ is not in a copy of $K_4$. Hence if $v$ is not admissible then, say, $xy, xz, xw \in E$. Let $N(y)=\{v,x,r,s\}$ and note that, since $v$ is not in a subgraph isomorphic to $K_4$, $r,s \in V-\{v,x,y,z,w\}$. By 4-regularity $vr, xs$ are not in $E$. Thus Claim \ref{claim:simple1} implies $y$ is admissible.
\end{proof}

\section{The 1-extension Operation}
\label{sec:hen2}

In this section we prove Theorem \ref{thm:Hen2admissible}. We start with some facts (Lemmas \ref{lem:connected} to \ref{lem:V3}) about $(2,1)$-circuits that can be proved by simple counting arguments. See \cite{B&J} or \cite{Nix} for similar results. Crucial to the problem is that we must retain simplicity throughout the recursive construction. We achieve this, in Subsection \ref{subsec:simple}, by establishing conditions for a $(2,1)$-circuit with $\delta(G)=3$ to have nodes not contained in subgraphs isomorphic to $K_4$. Finally in Subsection \ref{subsec:adm} we loosely follow the method established in \cite{B&J} to deduce admissibility for some node.

\subsection{Preliminaries}

Let $G=(V,E)$ be a $(2,1)$-circuit. A subset $X\subsetneq V$ is \emph{critical} if $i(X)=2|X|-1$ and is \emph{semi-critical} if $i(X)=2|X|-2$ and for all $X'\subseteq X$ we have $i(X')\leq 2|X'|-2$. We are particularly interested in special kinds of critical sets. We define a critical set $X\subsetneq V$ to be: \emph{$v$-critical} for a node $v\in V$ if $X$ contains exactly two neighbours of $v$ but not $v$ itself; \emph{node-critical} if $X$ is $v$-critical for some node $v$ (where $N(v)=\{x,y,z\}$) in $V$ such that $x,y\in X$ and $d(z)\geq 4$; and \emph{proper} if $|X|<|V|-1$. Note that node-critical sets are proper, but proper critical sets need not be node-critical.
Suppose that $A,B\subseteq V$; then we denote the number of edges $uv\in E$ such that $u\in A-B$ and $v\in B-A$ by $d(A,B)$.

\begin{lem}\label{lem:connected}
Let $G=(V,E)$ be a $(2,1)$-circuit. Then $G$ is connected and 2-edge-connected.
\end{lem}

\begin{lem}\label{lem:union}
Let $G=(V,E)$ be a $(2,1)$-circuit. Let $X,Y\subset V$ be critical sets, let $|X\cap Y|\geq 1$ and let $|X\cup Y|\leq |V|-1$. Then $X\cap Y$ and $X\cup Y$ are critical sets and $d(X,Y)=0$.
\end{lem}

\begin{lem}\label{lem:admissibleedge}
Let $G=(V,E)$ be a $(2,1)$-circuit. Let $v$ be a node with $N(v)=\{u,w,z\}$. Then removing $v$ and adding $uw$ is not admissible if and only if there exists a critical set $X\subset V$ with $u,w\in X$ and $v,z\notin X$ or $uw\in E$.
\end{lem}

\begin{lem}\label{lem:2edges}
Let $G=(V,E)$ be a 3-connected $(2,1)$-circuit containing a node $v$ with $N(v)=\{w,u,z\}$, $uz\notin E, wz,wu\in E$. Then $v$ is admissible.
\end{lem}

\begin{lem}\label{lem:1edge}
Let $G=(V,E)$ be a $(2,1)$-circuit containing a node $v$ with $N(v)=\{w,u,z\}$, $uz,wu\notin E, wz\in E$. Then $v$ is admissible.
\end{lem}

Let $V_3=\{v \in V:d(v)=3\}$. Let $V_3^*\subseteq V_3$ be the subset of degree three vertices not contained in copies of $K_4$.

\begin{lem}\label{lem:V3}
Let $G$ be a $(2,1)$-circuit. Then $G[V_3]$ is a (possibly empty) forest.
\end{lem}

\subsection{Finding Nodes}
\label{subsec:simple}

We now establish connectivity conditions that guarantee that, in a 3-connected, essentially 5-edge-connected $(2,1)$-circuit $G$ with $\delta(G)=3$, there exist nodes not contained in copies of $K_4$.
Since (the vertex set of) $K_4$ is semi-critical, this is significantly 
harder than the corresponding result in \cite{Nix}. We will first prove two lemmas for $(2,1)$-circuits that contain proper critical sets, before proving their (simpler) analogues for $(2,1)$-circuits that contain no proper critical sets. We make this distinction as we can give a weaker edge-connectivity assumption for $(2,1)$-circuits that contain proper critical sets, than for $(2,1)$-circuits that do not (see Theorem \ref{thm:Hen2admissible}).

\subsubsection{Proper critical sets} 

We first consider the case when there exists at least one proper critical set.

\begin{lem}
\label{lem:proper_implies_deg_3}
Let $G\in\mathcal{M}$ be essentially 4-edge-connected. Suppose that $G$ contains a proper critical set $X$. Then $|V_3-X|\geq 2$ (and hence $\delta(G)=3$).
\end{lem}

\begin{proof}
Let $E'$ be the set of edges between $X$ and $V-X$.
The sum of the degrees of vertices in $X$ is $4|X|-2+|E'|$. The total sum of the degrees of vertices in $G$ is $4|V|$. So the sum of the degrees in $V-X$  is $4(|V|-|X|)+2-|E'|$. Hence, as $|E'|\geq 4$, $\delta(G)\geq 3$ and the average degree of $G$ is four, there are at least two nodes that are not in $X$.
\end{proof}

\begin{lem}\label{lem:criticalsemi}
Let $G\in\mathcal{M}$ be 3-connected and essentially 4-edge-connected. 
Let $\W=\{W_i\mid 1\leq i\leq n:W_i\text{ is critical or semi-critical}\,\},$
where $W_1$ is critical and $|W_i|\geq 2$.  Let $Y=V-(\cup_{i=1}^n W_i)$. Suppose that either
\begin{enumerate}
\item $|Y|\geq 2$; or
\item $\A:=\cup_{i=1}^n G[W_i]$ is disconnected.
\end{enumerate}
Then $Y$ contains two nodes.
\end{lem}

\begin{proof}
Since $G\in \mathcal{M}$, with vertices of $V$ labelled as $v_1,v_2,\dots, v_{|V|}$, we have
\begin{eqnarray} \sum_{i=1}^{|V|}(4-d_G(v_i))=0.\label{eqn1} \end{eqnarray}
Let $A_1,\dots,A_m$ be the vertex sets of the connected components of $\A$. For any pair of semi-critical sets $W_i,W_j\in \mathcal{W}$ we have 
\begin{eqnarray} 2(|W_i|+|W_j|)-3 &=& 2|W_i\cup W_j|-1 + 2|W_i\cap W_j|-2 \nonumber\\ &\geq & i(W_i\cup W_j)+i(W_i\cap W_j)\nonumber\\ &=& i(W_i)+i(W_j)+d(W_i,W_j)\nonumber\\ &=& 2(|W_i|+|W_j|)-4+d(W_i,W_j)\label{eqnnew} \end{eqnarray}
so either $d(W_i,W_j)=0$ or equality holds and $d(W_i,W_j)=1$. 
Hence, whenever $W_i$ and $W_j$ have non-empty intersection, the union is semi-critical or critical. Similarly if $W_i$ is critical, $W_j$ is semi-critical and they have a non-empty intersection then the union is critical. (Note two critical sets with a non-empty intersection gives $V$, contradicting (1) and (2).)

Reordering if necessary, let $A_1,\dots, A_r$ be critical sets and let $A_{r+1},\dots, A_m$ be semi-critical sets. 
Define
\[ f_j=\sum_{v\in A_j}(4-d_{G[A_j]}(v)). \]
Hence, for $1\leq j\leq r$, $f_j=2$ and for $r+1\leq j\leq m$, $f_j=4$. As either (1) or (2) holds and there are no non-trivial 3-edge-cutsets, there exist four edges $x_\ell y_\ell$, $1\leq \ell \leq 4$, where $x_\ell \in A_j$ and $y_\ell \in V-A_j$, where the $x_\ell \in A_j$ (respectively the $y_\ell \in V-A_j$) are not necessarily all distinct. Define 
\[ g_j=\sum_{v\in A_j} (4-d_G(v)). \]
Hence, for $1\leq j \leq r$, $g_j\leq -2$ and, for $r+1\leq j \leq m$, $g_j\leq 0$. Now
\begin{eqnarray} \sum_{i=1}^m g_j = -2r\leq -2 \label{eqn2} \end{eqnarray}
since $r\geq 1$. The result follows by comparing Equations (\ref{eqn1}) and (\ref{eqn2}).
\end{proof}

\begin{lem}\label{lem:properlem}
Let $G=(V,E) \in\mathcal{M}$ be 3-connected and essentially 4-edge-connected. Suppose that $G$ contains a proper critical set $X$. Then $|V_3^*-X|\geq 2$.
\end{lem}

\begin{proof}
If $G$ contains no copies of $K_4$, then $V_3^*=V_3$ and Lemma \ref{lem:proper_implies_deg_3} yields the result.

So suppose that $B_1,\ldots, B_n$ are the vertex sets of all the copies of $K_4$ which are not subsets of $X$. Let $Y=V-(\cup_{i=1}^n B_i)-X$ and $\A=G[X]\cup(\cup_{i=1}^n G[B_i])$. If $|Y|\geq 2$ or $\A$ is disconnected, then we may apply Lemma \ref{lem:criticalsemi} to yield the result.

Hence, suppose that $|Y|\leq 1$ and $\A$ is connected. So, 
without loss of generality, there exists a $t\geq 1$, such that $X\cap B_i\neq \emptyset$, for all $1\leq i\leq t$.
For each of these $B_i$, and a subset $B'\subset B_i$, we have $i(B')=2|B'|-\alpha$ where $\alpha=2$ when $|B'|\in \{1,4\}$ and $\alpha=3$ when $|B'|\in \{2,3\}$. Thus, using a similar calculation to Equation (\ref{eqnnew}) we have: if $|X\cap B_i|=3$, then $X\cup B_i=V$, contradicting the fact that $X$ is a proper critical set, and if $|X\cap B_i|=2$, then $X\cup B_i=V$, contradicting 3-connectivity. So we have that $|X\cap B_i|=1$. If $d(X,B_i)>0$ or the multiplicity of any of the edges in $G[B_i]$ is greater than one, then $X\cup B_i=V$, contradicting 3-connectivity. So $d(X,B_i)=0$ and $G[B_i]$ is simple, for $1\leq i\leq t$.

We define an auxiliary graph $G^\dagger$ as follows. The vertex set is $X,B_1,\ldots,B_n$ and there is an edge between two vertices if and only if the corresponding sets in $G$ intersect.

\begin{claim}
The graph $G^\dagger$ is a tree.
\end{claim}

\begin{proof}
First note that a path in $G^\dagger$ either corresponds to a critical or a semi-critical set in $G$.

Suppose that $G^\dagger$ contains a cycle. Let $C=(c_0,c_1,\ldots, c_\ell)$ be such a cycle. Without loss of generality $c_0$ corresponds to one of the $B_i$, so let $P$ be the path obtained by deleting $c_0$ from $C$. 

Suppose $P$ corresponds to a critical set $P'=\cup_{i=1}^\ell c_i$ in $G$, that is $i(P')=2|P'|-1$. Now let $C'=\cup_{i=0}^\ell c_i$, note that $C'$ contains, at most, an additional two vertices from $P'$ and $G[C']$ contains an additional six edges from $G[P']$, a contradiction.

Suppose $P$ corresponds to a semi-critical set $P'=\cup_{i=1}^\ell c_i$ in $G$, that is $i(P')=2|P'|-2$. Then $X$ is not a vertex of $P$. Now let $C'=\cup_{i=0}^\ell c_i$, note that, as $P'$ is not critical, $G[C']$ contains an additional two vertices and an additional six edges from $G[P']$. Hence $C'=V$, but does not contain $X$, a contradiction.

The claim follows from the connectivity of $\A$.
\end{proof}

As $B_1$ exists it follows from the claim that $\A$ contains a cut-vertex, $v$. As $\A$ is connected the vertex set of $\A$ either corresponds to a critical set in $G$ or it is $V$. In the first case $v$ is part of a cut-pair in $G$ (as $|Y|\leq 1$) and in the second case $v$ is a cut-vertex in $G$. In either case we have contradicted the 3-connectivity of $G$.
\end{proof}

\subsubsection{No proper critical sets}

Now we deal with the case when there are no proper critical sets.

\begin{lem}\label{lem:noproper}
Let $G=(V,E)\in\mathcal{M}$ be essentially 5-edge-connected. Suppose $V$ contains no proper critical sets, let $W_1,\dots, W_k$, with $k\geq 1$, be the vertex sets of all the copies of $K_4$ in $G$ and let $Y=V-\cup_{i=1}^k W_i$. Suppose that either
\begin{enumerate}
\item $|Y|\geq 2$; or
\item $\cup_{i=1}^kG[W_i]$ is disconnected.
\end{enumerate}
Then $G$ contains an admissible node.
\end{lem}

\begin{proof}
If any of the $W_i$ are critical sets (i.e., $G[W_i]$ contains multi-edges or loops), then, as $V$ contains no proper critical sets, $|V|=5$ and neither (1) nor (2) hold.
So, since $G$ is a $(2,1)$-circuit, Equation (\ref{eqn1}) holds and $W_i$ is semi-critical, for each $1\leq i \leq k$. Let $A_1,\dots,A_r$ be the vertex sets of the connected components of $\cup_{i=1}^k G[W_i]$. 
Observe that each of $A_1,\dots,A_m$ is semi-critical.

Then
\[ 4=\sum_{v\in A_j}(4-d_{G[A_j]}(v)). \]
There are no non-trivial 4-edge-cutsets so, as at least one of (1) or (2) holds, there exists five edges $x_\ell y_\ell$, $1\leq \ell \leq 5$, where $x_\ell \in A_j$ and $y_\ell \in V-A_j$ (where the $x_\ell$ are not necessarily distinct and nor are the $y_\ell$). Then 
\begin{eqnarray} \sum_{j=1}^r \sum_{v\in A_j} (4-d_G(v))\leq -1. \label{eqn3} \end{eqnarray}
Comparing Equations (\ref{eqn1}) and (\ref{eqn3}) gives that $|Y|$ contains a node $v$. Now $v$ is not in a copy of $K_4$ and since there are no proper critical sets, we can apply Lemma \ref{lem:admissibleedge} to deduce that $v$ is admissible.
\end{proof}

\subsection{Admissibility}
\label{subsec:adm}

\begin{lem}
\label{lem:silly}
Let $G=(V,E)$ be a 3-connected $(2,1)$-circuit. Suppose $v\in V_3^*$ is a non-admissible node with $N(v)=\{x,y,z\}$. Then there exists two $v$-critical sets $X,Y$ such that $X\cup Y=V-v$. Moreover, we may choose $X,Y$ such that $z\in X\cap Y$.
\end{lem}

\begin{proof}
By Lemmas \ref{lem:2edges} and \ref{lem:1edge}, since $v$ is not admissible and $v\in V_3^*$, there are no edges between $x$, $y$ and $z$. Now \ref{lem:admissibleedge} implies there exists critical sets $X$ containing $x,z$ but not $y,v$ and $Y$ containing $y,z$ but not $x,v$. 
It follows from Lemma \ref{lem:union} and the definition of a $(2,1)$-circuit that $X\cup Y$ is critical, hence $X\cup Y=V-v$.
\end{proof}

The following has essentially the same proof as the corresponding results for $(2,3)$ and $(2,2)$-circuits, \cite[Lemma $3.3$]{B&J} and \cite[Lemma $2.10$]{Nix}; however we obtain a slightly stronger result.

\begin{lem}\label{lem:leaf}
Let $G=(V,E)$ be a 
$(2,1)$-circuit with $\delta(G)=3$. Let $v \in V$ be a node with $N(v)=\{x,y,z\}$, $d(z)\geq 4$ and suppose there are no edges between neighbours of $v$. Let $X$ be $v$-critical on $x,y$. Suppose either
\begin{enumerate}
\item there is a non-admissible node $u\in V-X-v$ where $d_{G[V_3^*]}(u)=2$ with no edges between its neighbours, precisely one neighbour $w$ in $X$ and $w$ is a node; or
\item there is a non-admissible node $u \in V-X-v$ where $d_{G[V_3^*]}(u)=0\text{ or }1$ with no edges between its neighbours. 
\end{enumerate}
Then, in $G$, there is a node-critical set $X'$ for a node with no edges between its neighbours such that $X\subsetneq X'$.
\end{lem}

\begin{proof}
Suppose (1) holds and let $N(u)=\{w,p,q\}$. 
Since $u$ is non-admissible and $wp\not\in E$, by Lemma \ref{lem:admissibleedge}, there exists a $u$-critical set, $Y$ say, on $w$ and $p$. Now, $w\in X\cup Y$, and $u,q\not\in X\cup Y$, so by Lemma \ref{lem:union}, $X':=X\cup Y$ is node-critical for $u$. As $p\not\in X$, $X\subsetneq X'$.

So, suppose that (2) holds. Let $r$ be a neighbour of $u$ of minimum degree in $G$, using the argument in the proof of Lemma \ref{lem:silly}, we see that there exist $u$-critical sets $Y_1$ and $Y_2$ whose intersection contains $r$ and $Y_1\cup Y_2= V-u$. Hence, without loss of generality, $X\cap Y_1\neq \emptyset$. As $|N(u)\cap X|\leq 3$, we consider four cases. If $|N(u)\cap X|=0$, then the set $X'=X\cup Y_1$ is node-critical and, as $|N(u)\cap X|=0$, we have that $X\subsetneq X'$.

Suppose that $|N(u)\cap X|=1$, say $N(u)\cap X=\{s\}$. If $s\in Y_1$, then set $X'=X\cup Y_1$ is node-critical and, as $N(u)-(X\cup Y_1)\neq \emptyset$, we have that $X\subsetneq X'$. So, assume that $s\not\in Y_1$, that is $s\in Y_2-Y_1$; hence $X'=X\cup Y_2$ is node-critical and, as $N(u)-(X\cup Y_2)\neq \emptyset$, we have that $X\subsetneq X'$. 

If $|N(u)\cap X|=2$, then the set $X'=X\cup \{u\}$ is node-critical and hence $X\subsetneq X'$. Finally, if $|N(u)\cap X|=3$, then $V=X\cup \{u\}$, but $v\not\in X\cup \{u\}$, a contradiction.
\end{proof}

Before we prove Theorem \ref{thm:Hen2admissible} we establish the following lemma to deal with some small circuits. Note that the set of graphs $\mathcal{G}^*$ is comprised of the graphs in the set $\mathcal{G} = \{K_5,\overline{G57},\overline{G59},\break\overline{G293},\overline{G308},\overline{G312}, S_1, S_2,S_3,S_4,S_5\}$  and those shown in Figure \ref{fig:G*stuff}.

\begin{figure}
\begin{center}
\begin{tikzpicture}[scale=1, vertex/.style={circle,inner sep=2,fill=black,draw}, vertex2/.style={circle,inner sep=4,fill=black,draw}]

\coordinate (a1) at (1.75,7.5);

\coordinate (c1) at (3.75,7);
\coordinate (c2) at (4.25,8);
\coordinate (c3) at (4.75,7);

\coordinate (d1) at (6.25,7);
\coordinate (d2) at (6.75,8);
\coordinate (d3) at (7.25,7);

\coordinate (e1) at (8.75,7);
\coordinate (e2) at (9.25,8);
\coordinate (e3) at (9.75,7);


\coordinate (u1) at (0,4);
\coordinate (u2) at (1,4);
\coordinate (u3) at (0,5);
\coordinate (u4) at (1,5);

\coordinate (v1) at (2.5,4);
\coordinate (v2) at (3.5,4);
\coordinate (v3) at (2.5,5);
\coordinate (v4) at (3.5,5);

\coordinate (w1) at (5,4);
\coordinate (w2) at (6,4);
\coordinate (w3) at (5,5);
\coordinate (w4) at (6,5);

\coordinate (x1) at (7.5,4);
\coordinate (x2) at (8.5,4);
\coordinate (x3) at (7.5,5);
\coordinate (x4) at (8.5,5);

\coordinate (y1) at (10,4);
\coordinate (y2) at (11,4);
\coordinate (y3) at (10,5);
\coordinate (y4) at (11,5);


\coordinate (m1) at (0.5,0.5);
\coordinate (m2) at (0,1.5);
\coordinate (m3) at (1,2.25);
\coordinate (m4) at (2,1.5);
\coordinate (m5) at (1.5,0.5);

\coordinate (n1) at (3.5,0.5);
\coordinate (n2) at (3,1.5);
\coordinate (n3) at (4,2.25);
\coordinate (n4) at (5,1.5);
\coordinate (n5) at (4.5,0.5);

\coordinate (o1) at (6.5,0.5);
\coordinate (o2) at (6,1.5);
\coordinate (o3) at (7,2.25);
\coordinate (o4) at (8,1.5);
\coordinate (o5) at (7.5,0.5);

\coordinate (p1) at (9.5,0.5);
\coordinate (p2) at (9,1.5);
\coordinate (p3) at (10,2.25);
\coordinate (p4) at (11,1.5);
\coordinate (p5) at (10.5,0.5);


\draw (1.75,7.65) circle (0.15cm);
\draw (1.75,7.35) circle (0.15cm);

\draw[bend left=15] (c1) edge (c2);
\draw[bend left=15] (c2) edge (c1);
\draw[bend left=15] (c3) edge (c2);
\draw[bend left=15] (c2) edge (c3);
\draw[bend left=15] (c1) edge (c3);
\draw[bend left=15] (c3) edge (c1);

\draw (6.75,8.15) circle (0.15cm);
\draw (d1) -- (d2) -- (d3) -- cycle;
\draw[bend left=20] (d1) edge (d3);
\draw[bend left=20] (d3) edge (d1);

\draw (e1) -- (e2) -- (e3) -- cycle;
\draw (9.25,8.15) circle (0.15cm);
\draw (8.64393398283,6.89393398282) circle (0.15cm); 
\draw (9.85606601717,6.89393398282) circle (0.15cm); 


\draw (u1) -- (u2) -- (u3) -- (u4) -- (u1) -- (u3);
\draw (u4) -- (u2);
\draw (-0.10606601717,5.10606601717) circle (0.15cm); 
\draw (1.10606601717,5.10606601717) circle (0.15cm); 

\draw (v1) -- (v2) -- (v3) -- (v4) -- (v1);
\draw (v4) -- (v2);
\draw[bend left=20] (v1) edge (v3);
\draw[bend left=20] (v3) edge (v1);
\draw (3.60606601717,5.10606601717) circle (0.15cm); 

\draw (w1) -- (w2) -- (w3) -- (w4) -- (w1);
\draw (w4) -- (w2);
\draw[bend left=20] (w1) edge (w3);
\draw[bend left=20] (w3) edge (w1);
\draw (w3) edge (w1);

\draw (x1) -- (x2);
\draw (x4) -- (x1);
\draw (x3) -- (x2);
\draw (x3) -- (x4);
\draw[bend left=20] (x1) edge (x3);
\draw[bend left=20] (x3) edge (x1);
\draw[bend left=20] (x2) edge (x4);
\draw[bend left=20] (x4) edge (x2);

\draw (y1) -- (y2);
\draw (y4) -- (y1);
\draw (y3) -- (y2);
\draw (y2) -- (y4);
\draw[bend left=20] (y1) edge (y3);
\draw[bend left=20] (y3) edge (y1);
\draw[bend left=20] (y3) edge (y4);
\draw[bend left=20] (y4) edge (y3);


\draw (m5) -- (m4) -- (m3) -- (m2) -- (m1) -- (m4) -- (m2) -- (m5) -- (m3) -- (m1);
\draw (1,2.4) circle (0.15cm);

\draw (n5) -- (n4) -- (n3) -- (n2) -- (n1) -- (n4) -- (n2) -- (n5) -- (n3) -- (n1);
\draw (3.39393398282,0.39393398282) circle (0.15cm); 

\draw (o5) -- (o4);
\draw (o3) -- (o2) -- (o1) -- (o4) -- (o2) -- (o5) -- (o3) -- (o1);
\draw[bend left=20] (o4) edge (o3);
\draw[bend left=20] (o3) edge (o4);

\draw (p4) -- (p3) -- (p2) -- (p1) -- (p4) -- (p2) -- (p5) -- (p3) -- (p1);
\draw[bend left=20] (p4) edge (p5);
\draw[bend left=20] (p5) edge (p4);


\node at (a1) [vertex]{};

\node at (c1) [vertex]{};
\node at (c2) [vertex]{};
\node at (c3) [vertex]{};

\node at (d1) [vertex]{};
\node at (d2) [vertex]{};
\node at (d3) [vertex]{};

\node at (e1) [vertex]{};
\node at (e2) [vertex]{};
\node at (e3) [vertex]{};


\node at (u1) [vertex]{};
\node at (u2) [vertex]{};
\node at (u3) [vertex]{};
\node at (u4) [vertex]{};

\node at (v1) [vertex]{};
\node at (v2) [vertex]{};
\node at (v3) [vertex]{};
\node at (v4) [vertex]{};

\node at (w1) [vertex]{};
\node at (w2) [vertex]{};
\node at (w3) [vertex]{};
\node at (w4) [vertex]{};

\node at (x1) [vertex]{};
\node at (x2) [vertex]{};
\node at (x3) [vertex]{};
\node at (x4) [vertex]{};

\node at (y1) [vertex]{};
\node at (y2) [vertex]{};
\node at (y3) [vertex]{};
\node at (y4) [vertex]{};


\node at (m1) [vertex]{};
\node at (m2) [vertex]{};
\node at (m3) [vertex]{};
\node at (m4) [vertex]{};
\node at (m5) [vertex]{};

\node at (n1) [vertex]{};
\node at (n2) [vertex]{};
\node at (n3) [vertex]{};
\node at (n4) [vertex]{};
\node at (n5) [vertex]{};

\node at (o1) [vertex]{};
\node at (o2) [vertex]{};
\node at (o3) [vertex]{};
\node at (o4) [vertex]{};
\node at (o5) [vertex]{};

\node at (p1) [vertex]{};
\node at (p2) [vertex]{};
\node at (p3) [vertex]{};
\node at (p4) [vertex]{};
\node at (p5) [vertex]{};


\node at (1.75,6.5) {$R_0$};

\node at (4.25,6.5) {$R_1$};
\node at (6.75,6.5) {$R_2$};
\node at (9.25,6.5) {$R_3$};


\node at (0.5,3.5) {$R_4$};
\node at (3,3.5) {$R_5$};
\node at (5.5,3.5) {$R_{6}$};
\node at (8,3.5) {$R_{7}$};
\node at (10.5,3.5) {$R_{8}$};


\node at (1,0) {$R_{9}$};
\node at (4,0) {$R_{10}$};
\node at (7,0) {$R_{11}$};
\node at (10,0) {$R_{12}$};

\end{tikzpicture}
\end{center}
\caption{Circuits in $\mathcal{G}^*\cap(\mathcal{M}\setminus M(2,1))$.}
\label{fig:G*stuff}
\end{figure}
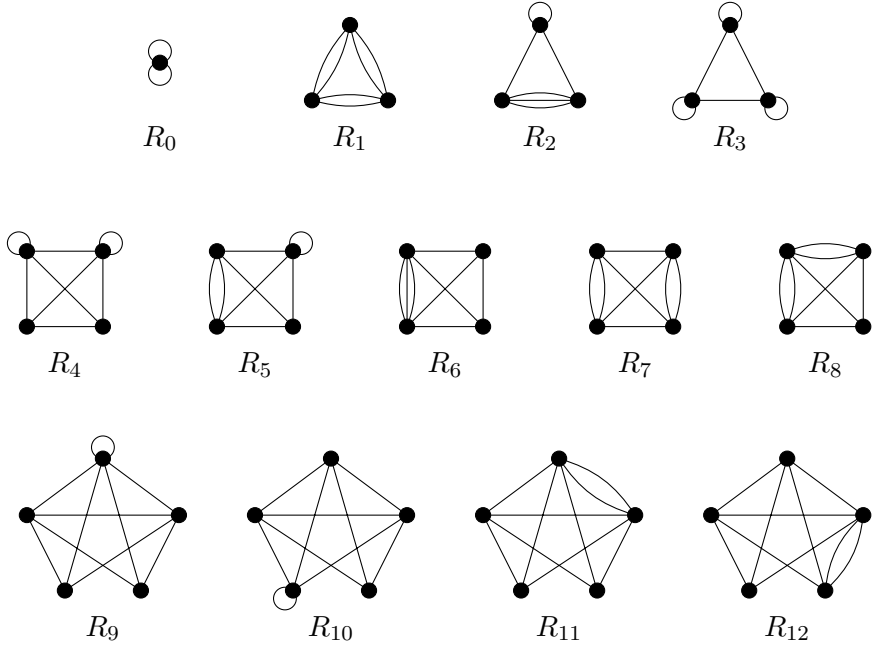

\begin{lem}
\label{lem:small_ones}
Let $G=(V,E)\in\mathcal{M}\setminus\mathcal{G}^*$ be 3-connected and essentially 5-edge-connected  with $|V|\leq 8$, $\delta(G)=3$ and no proper critical sets. 

Suppose that $G$ contains a copy of $K_4$. Let $W_1,\ldots,W_k$, with $k\geq 1$, be the vertex sets of all copies of $K_4$ in $G$ and let $Y=V-\cup_{i=1}^k W_i$. Further suppose that $\mathcal{A}=\cup_{i=1}^k G[W_i]$ is connected and $|Y|\leq 1$. Then $G$ contains an admissible node.
\end{lem}

\begin{proof}
First note that, as $G$ contains a copy of $K_4$, $|V|\geq 4$.

Suppose that $|V|=4$. In this case $V=W_1$ and as $G$ is a circuit, $G\in\mathcal{M}'$ (so $G$ is a multigraph). There are five possibilities for a circuit on four vertices: $R_4$, $R_5$, $R_6$, $R_7$, $R_8$ (see Figure \ref{fig:G*stuff}) all of which are contained in $\mathcal{G}^*$.  

Suppose that $|V|=5$. Either there are two vertex sets underlying copies of $K_4$ or there is one set and an additional vertex of degree three. Note that these cases are equivalent, hence there are five possibilities for a such a circuit on five vertices: $K_5$, $R_9$, $R_{10}$, $R_{11}$, $R_{12}$ (see Figure \ref{fig:G*stuff}) all of which are contained in $\mathcal{G}^*$.

Suppose that $|V|=6$. As $|Y|\leq 1$, $k\geq 2$ and any pair of distinct $W_i$, $W_j$ intersect in either two or three vertices. Consider $W_1$ and $W_2$. If $|W_1\cap W_2|=2$, then $V=W_1\cup W_2$, and as $G$ is 3-connected it must be isomorphic to $\overline{G59}\in\mathcal{G}$. So suppose that $|W_1\cap W_2|=3$. Then there exists a vertex $v\not\in W_1\cup W_2$. As $W_1\cup W_2$ is a critical set $v$ has degree three and hence $G$ is either isomorphic to $\overline{G57}$ or $\overline{G59}$, in either case $G\in\mathcal{G}$.

Suppose that $|V|= 7$. As $|Y|\leq 1$ and $\mathcal{A}$ is connected, $k\geq 2$ and any two $W_i$, $W_j$ intersect in either one, two or three vertices. However, as $G$ does not contain a proper critical set $|W_i\cap W_j|\neq 3$. 
If $|W_1\cap W_2|=1$, then as $G$ is 3-connected it must be isomorphic to $\overline{G312}$. 
So suppose that $|W_1\cap W_2|=2$. Then there exists a vertex $v\not\in W_1\cup W_2$. As $W_1\cup W_2$ is a critical set $v$ has degree three and as $G$ is 3-connected it is isomorphic to one of the two graphs illustrated in Figure \ref{fig:order7_admiss}. In either case Lemma \ref{lem:1edge} or \ref{lem:2edges} implies the existence of an admissible node.

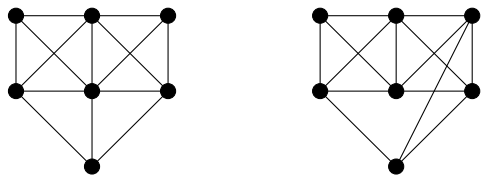
\begin{figure}
\begin{center}
\begin{tikzpicture}[scale=1, vertex/.style={circle,inner sep=2,fill=black,draw}, vertex2/.style={circle,inner sep=4,fill=black,draw}]

\coordinate (v1) at (0,0);
\coordinate (v2) at (1,0);
\coordinate (v3) at (2,0);
\coordinate (v4) at (0,1);
\coordinate (v5) at (1,1);
\coordinate (v6) at (2,1);
\coordinate (v7) at (1,-1);

\draw (v1) -- (v3) -- (v6) -- (v4) -- (v1) -- (v5) -- (v2) -- (v6);
\draw (v4) -- (v2);
\draw (v5) -- (v3);

\draw (v1) -- (v7) -- (v2);
\draw (v7) -- (v3);

\node at (v1) [vertex]{};
\node at (v2) [vertex]{};
\node at (v3) [vertex]{};
\node at (v4) [vertex]{};
\node at (v5) [vertex]{};
\node at (v6) [vertex]{};
\node at (v7) [vertex]{};


\coordinate (u1) at (4,0);
\coordinate (u2) at (5,0);
\coordinate (u3) at (6,0);
\coordinate (u4) at (4,1);
\coordinate (u5) at (5,1);
\coordinate (u6) at (6,1);
\coordinate (u7) at (5,-1);

\draw (u1) -- (u3) -- (u6) -- (u4) -- (u1) -- (u5) -- (u2) -- (u6);
\draw (u4) -- (u2);
\draw (u5) -- (u3);

\draw (u1) -- (u7) -- (u3);
\draw (u7) -- (u6);

\node at (u1) [vertex]{};
\node at (u2) [vertex]{};
\node at (u3) [vertex]{};
\node at (u4) [vertex]{};
\node at (u5) [vertex]{};
\node at (u6) [vertex]{};
\node at (u7) [vertex]{};

\end{tikzpicture}
\end{center}
\caption{Potential base graphs on seven vertices.}
\label{fig:order7_admiss}
\end{figure}

Finally, suppose that $|V|=8$. As $\mathcal{A}$ is connected and $|Y|\leq 1$, $k\geq 2$ and for any pair $W_i$, $W_j$, where $1\leq i<j\leq k$, $|W_i\cap W_j|=0\text{ or }1$ and there exists a pair, say $W_1$, $W_2$, where $|W_1\cap W_2|=1$. Then there exists a vertex $v\not\in W_1\cup W_2$ and as $G$ is 3-connected in this case it must be simple and isomorphic to one of the three graphs illustrated in Figure \ref{fig:order8_admiss}; in each case Lemma \ref{lem:1edge} or \ref{lem:2edges} implies the existence of an admissible node. 
\end{proof}

\begin{figure}
\begin{center}
\begin{tikzpicture}[scale=1, vertex/.style={circle,inner sep=2,fill=black,draw}, vertex2/.style={circle,inner sep=4,fill=black,draw}]

\coordinate (v1) at (0,0);
\coordinate (v2) at (1,0);
\coordinate (v3) at (0,1);
\coordinate (v4) at (1,1);
\coordinate (v5) at (2,0);
\coordinate (v6) at (1,-1);
\coordinate (v7) at (2,-1);
\coordinate (v8) at (2.5,1);

\draw (v1) -- (v3) -- (v2) -- (v4) -- (v1) -- (v6) -- (v7) -- (v5) -- (v6) -- (v2) -- (v5) -- (v8) -- (v2) -- (v1);
\draw (v3) -- (v8);
\draw (v2) -- (v7);

\node at (v1) [vertex]{};
\node at (v2) [vertex]{};
\node at (v3) [vertex]{};
\node at (v4) [vertex]{};
\node at (v5) [vertex]{};
\node at (v6) [vertex]{};
\node at (v7) [vertex]{};
\node at (v8) [vertex]{};


\coordinate (u1) at (4,0);
\coordinate (u2) at (5,0);
\coordinate (u3) at (4,1);
\coordinate (u4) at (5,1);
\coordinate (u5) at (6,0);
\coordinate (u6) at (5,-1);
\coordinate (u7) at (6,-1);
\coordinate (u8) at (6.5,1);

\draw (u1) -- (u3) -- (u2) -- (u4) -- (u1) -- (u6) -- (u7) -- (u5) -- (u6) -- (u2) -- (u5) -- (u8) -- (u7);
\draw (u3) -- (u8);

\draw (u2) -- (u7);
\draw (u2) -- (u1);

\node at (u1) [vertex]{};
\node at (u2) [vertex]{};
\node at (u3) [vertex]{};
\node at (u4) [vertex]{};
\node at (u5) [vertex]{};
\node at (u6) [vertex]{};
\node at (u7) [vertex]{};
\node at (u8) [vertex]{};


\coordinate (w1) at (8,0);
\coordinate (w2) at (9,0);
\coordinate (w3) at (8,1);
\coordinate (w4) at (9,1);
\coordinate (w5) at (10,0);
\coordinate (w6) at (9,-1);
\coordinate (w7) at (10,-1);
\coordinate (w8) at (10.5,1);

\draw (w1) -- (w3) -- (w2) -- (w4) -- (w1);
\draw (w6) -- (w7) -- (w5) -- (w6)-- (w2) -- (w5) -- (w8) -- (w7);
\draw (w3) -- (w4);
\draw (w4) -- (w5);
\draw (w2) -- (w7);
\draw (w2) -- (w1);

\draw[bend left=30] (w3) edge (w8);

\node at (w1) [vertex]{};
\node at (w2) [vertex]{};
\node at (w3) [vertex]{};
\node at (w4) [vertex]{};
\node at (w5) [vertex]{};
\node at (w6) [vertex]{};
\node at (w7) [vertex]{};
\node at (w8) [vertex]{};

\end{tikzpicture}
\end{center}
\caption{Potential base graphs on eight vertices.}
\label{fig:order8_admiss}
\end{figure}
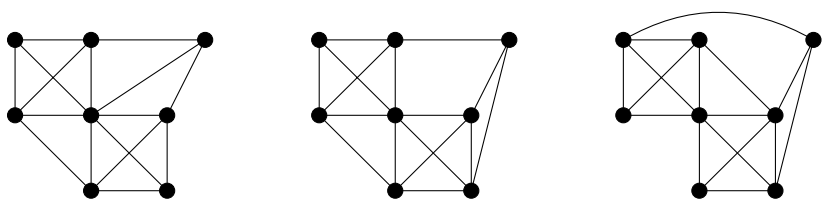

We can now prove Theorem \ref{thm:Hen2admissible}.
\begin{thm1.3}
Let $G=(V,E)\in \mathcal{M}$ be $3$-connected with $\delta(G)=3$ and $G\not\in \G^\ast$. Suppose either
\begin{enumerate}
\item[(i)] $G$ is essentially 5-edge-connected and there is no proper critical set or
\item[(ii)] $G$ is essentially 4-edge-connected and there is a proper critical set.
\end{enumerate}
Then $G$ contains an admissible node.
\end{thm1.3}

\begin{proof}

Suppose that (i) holds.
There exists a vertex $v$ of degree three. If $v$ is not contained in a copy of $K_4$, then, as there are no proper critical sets, we can apply Lemma \ref{lem:admissibleedge} to show that $v$ is admissible. 

So we may assume that $G$ contains a copy of $K_4$. 
Let $W_1,\dots, W_k$, with $k\geq 1$, be the vertex sets of all the copies of $K_4$ in $G$ and let $Y=V-\cup_{i=1}^k W_i$. 
By Lemma \ref{lem:noproper} we may assume that $\mathcal{A}=\cup_{i=1}^k G[W_i]$ is connected and $|Y|\leq 1$. 
If $G$ has less than 9 vertices, then, by Lemma \ref{lem:small_ones}, it is either in $\G^*$ or has an admissible node.

So suppose $G$ has at least 9 vertices. Since there are no proper critical sets, the only options for two copies of $K_4$, with vertex sets $W_i$ and $W_j$ say, are that $W_i\cap W_j=\emptyset$ or $|W_i\cap W_j|=1$ and $d(W_i,W_j)=0$. 

We define an auxiliary graph $G^\ddagger$ as follows. The vertex set is $W_1,\ldots, W_n$ and there is an edge between two vertices if and only if the corresponding sets in $G$ intersect.

\begin{claim}
The graph $G^\ddagger$ is either a tree or a cycle.
\end{claim}

\begin{proof}
We show that if $G^\ddagger$ contains a cycle this must be the whole graph; then, as $\A$ is connected, the claim follows.
As $W_i\cap W_j=\emptyset$ or $|W_i\cap W_j|=1$ and $d(W_i,W_j)=0$ for $1\leq i<j\leq n$ a path in $G^\ddagger$ corresponds to a semi-critical set in $G$.

Suppose that $G^\ddagger$ contains a cycle. Let $C=(c_0,c_1,\ldots, c_\ell)$ and let $P$ be the path obtained by deleting $c_0$ from $C$. Consider the set of vertices $P'=\cup_{i=1}^\ell c_i$ in $G$, then $i(P')=2|P'|-2$. Now let $C'=\cup_{i=1}^\ell c_i$ and note that $G[C']$ contains an additional two vertices and an additional six edges from $G[P']$. Hence $C'=V$.
\end{proof}

If $G^\ddagger$ is a cycle, then $i(\cup_{i=1}^kW_i)=2|\cup_{i=1}^kW_i|$ so we must have $|Y|=0$ and every edge of $G$ is contained in one of the $K_4$'s so $G$ has a cut-pair, a contradiction. Hence, $G^\ddagger$ is a tree. 

As $|V|\geq 9$ and $\mathcal{A}=\cup_{i=1}^k G[W_i]$ is connected, $G^\ddagger$ has at least three vertices. Note that $2|\cup_{i=1}^kW_i|-2\leq i(\cup_{i=1}^kW_i)\leq 2|\cup_{i=1}^kW_i|-1$. In either case we necessarily have $|Y|=1$. Denote the vertex in $Y$ as $v$.
If $\cup_{i=1}^k W_i$ is a semi-critical set in $G$ then $d(v)=4$. As $G^\ddagger$ is a tree, $G$ contains a cut-pair (the vertex $v$ and one of the vertices in the intersection of a pair $W_i$, $W_j$ that intersect); a contradiction. Hence we may suppose that $\cup_{i=1}^k W_i$ is a critical set. Then $d(v)=3$. 
Since $\cup_{i=1}^k W_i$ is critical we have $d(W_i,W_j)=1$ for some pair $i,j$ and $d(W_k,W_\ell)=0$ for all other pairs. If we regard this edge as an edge $W_iW_j$ in $G^\ddagger$ then we get a cycle. If this cycle does not contain every vertex of $G^\ddagger$ then $G$ has a proper critical set, contrary to our assumption.
Hence, by reordering if necessary, we may assume that $G^\ddagger$ is a path on $W_1,W_2,\dots,W_k$ and that $d(W_1,W_k)=1$. It is now easy to find a cutpair in $G$ if $|V(G^\ddagger)|\geq 4$ (thereby contradicting 3-connectivity). Hence, as $G^\ddagger$ contains at least three vertices, $|V(G^\ddagger)|=3$ and evidently $v$ is admissible.

Now, suppose that (ii) holds. 
Let $\X=\{X\subseteq V\mid X\text{ is a node-critical set in }G\}$.
Suppose that $\X=\emptyset$. By Lemma \ref{lem:properlem}, $|V_3^*|\geq 2$ and, by Lemma \ref{lem:V3}, $G[V_3^*]$ is a forest. Let $v$ be a vertex of $G[V_3^*]$.
As $v$ is not contained in a copy of $K_4$ and $\X=\emptyset$, by Lemma \ref{lem:admissibleedge}, $v$ must be admissible. 

So we may assume that $\X\neq\emptyset$. By Lemma \ref{lem:properlem} we have $|V_3^*|\geq 2$. By Lemmas \ref{lem:1edge} and \ref{lem:2edges}, we may assume that each vertex in $V_3^*$ has no edges between the neighbours.

Choose $X\in\X$ to be a maximal node-critical set chosen over all nodes in $V_3^*$. Suppose $X$ is node-critical for $v$ on $x,y$ where $N(v)=\{x,y,z\}$ and $d_G(z)\geq 4$. By Lemma \ref{lem:properlem}, $V-X-v$ contains a node $u\in V_3^*$ and, by Lemma \ref{lem:V3}, we may choose $u$ to be a leaf in $G[V_3^*-X-v]$.

By the maximality of $X$, each vertex $t\in V-X-v-z$ has at most one neighbour in $X$. Hence $d_{G[V_3^*]}(u)\leq 2$, so $d_{G[V_3]}(u)\leq 2$. 
If $d_{G[V_3^*]}(u)=2$, then, since $u$ has precisely one neighbour $w$ in $X$ and since $u$ is a leaf in $G[V_3^*-X-v]$, it follows that $w$ is a node. Thus Lemma \ref{lem:leaf} Part (1) and the maximality of $|X|$ imply that $u$ is an admissible node.
If $d_{G[V_3^*]}(u)=0\text{ or }1$, then Lemma \ref{lem:leaf} Part (2) and the maximality of $|X|$ imply that $u$ is admissible. 
\end{proof}

\section{Sum Moves}
\label{sec:sumsec}

In this section we reduce the connectivity assumptions from the previous sections by defining operations that `pull apart' $(2,1)$-circuits into two smaller $(2,1)$-circuits.

Consider a graph $G=(V,E)$ that is either not 3-connected or has a non-trivial 4-edge-cutset. We will provide a series of lemmas that consider the cases where $G$ satisfies one of the following:
\begin{enumerate}
\item a cut-vertex $x$, $A,B\subsetneq V$ such that $A\cap B=\{x\}$, $A\cup B=V$, $i(A)=2|A|-1$ and $i(B)=2|B|-1$;
\item a cut-pair $x,y$ (where neither are cut-vertices), $xy\notin E$, $A,B\subsetneq V$ such that $A\cap B=\{x,y\}$, $A\cup B=V$ and  either
\begin{enumerate}
\item $i(A)=2|A|-2$ and $i(B)=2|B|-2$, or
\item $i(A)=2|A|-1$ and $i(B)=2|B|-3$;
\end{enumerate}
\item a cut-pair $x,y$ (where neither are cut-vertices), $xy\in E$, $A,B\subsetneq V$ such that $A\cap B=\{x,y\}$, $A\cup B=V$, $i(A)=2|A|-1$ and $i(B)=2|B|-2$;
\item a 3-edge-cutset $\{x_iy_i: x_i\in A, y_i \in B, 1\leq i \leq 3\}$, where $|\{x_i,y_i\mid 1\leq i\leq 3\}|=6$, for $A,B\subsetneq V$, such that $A\cap B=\emptyset$,  $A\cup B=V$, $i(A)=2|A|-1$ and $i(B)=2|B|-2$; or
\item a 4-edge-cutset $\{x_iy_i: x_i\in A, y_i \in B, 1\leq i \leq 4\}$ for $A,B\subsetneq V$, such that $A\cap B=\emptyset$, $A\cup B=V$, $i(A)=2|A|-2$ and $i(B)=2|B|-2$.
\end{enumerate}

\begin{lem}\label{lem:sum1}
Let $G=(V,E)$ be a graph satisfying (1) in the list above. Let $a_i,b_i,c_i,d_i$, for $i=A,B$ be eight distinct vertices not in $V$. Then $G$ is a $(2,1)$-circuit if and only if the graphs $G_i=G[i]\cup (K_5(a_i,b_i,c_i,d_i,x)-e)$, for $i=A,B$, are $(2,1)$-circuits. (See Figure \ref{fig:sum1}.)
\end{lem}

\begin{figure}
\begin{center}
\begin{tikzpicture}[scale=1, fill=gray!30, vertex/.style={circle,inner sep=2,fill=black,draw}]

\filldraw[thick] (1,1) ellipse (1 and 0.8);
\node at (1,0.5) [label=north:$2|A|-1$]{};

\filldraw[thick] (3,1) ellipse (1 and 0.8);
\node at (3,0.5) [label=north:$2|B|-1$]{};

\node at (2,1) [vertex]{};

\draw[ultra thick,<->](4.5,1) -- (5,1);

\filldraw[thick] (6.5,1) ellipse (1 and 0.8);
\node at (6.5,0.5) [label=north:$2|A|-1$]{};

\draw (7.5,1) -- (8.5,0.75) -- (8,1.5) -- (8,0.5) -- (8.5,1.25) -- cycle;
\draw (8.5,1.25) -- (8,1.5) -- (7.5,1) -- (8,0.5) -- (8.5,0.75);

\node at (7.5,1) [vertex]{};
\node at (8.5,0.75) [vertex]{};
\node at (8.5,1.25) [vertex]{};
\node at (8,1.5) [vertex]{};
\node at (8,0.5) [vertex]{};

\filldraw[thick] (11,1) ellipse (1 and 0.8);
\node at (11,0.5) [label=north:$2|B|-1$]{};

\draw (10,1) -- (9,0.75) -- (9.5,1.5) -- (9.5,0.5) -- (9,1.25) -- cycle;
\draw (9,1.25) -- (9.5,1.5) -- (10,1) -- (9.5,0.5) -- (9,0.75);

\node at (10,1) [vertex]{};
\node at (9,0.75) [vertex]{};
\node at (9,1.25) [vertex]{};
\node at (9.5,1.5) [vertex]{};
\node at (9.5,0.5) [vertex]{};

\end{tikzpicture}
\end{center}
\caption{Illustration of the sum move for Case (1).}
\label{fig:sum1}
\end{figure}
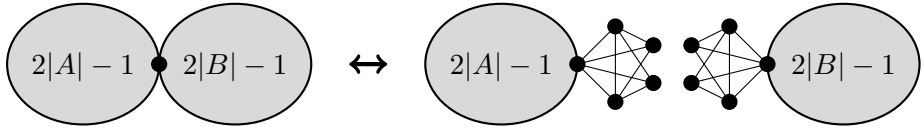

Note that the choice of the missing edge in the $K_5-e$ is arbitrary.

\begin{proof}
Suppose $G_A$ and $G_B$ are $(2,1)$-circuits. Assume for a contradiction that $G$ is not. Then there is a proper subset $X\subsetneq V$ with $i(X)=i(A\cap X)+i(B\cap X)\geq 2|X|$. However $i(A\cap X)\leq 2|A\cap X|-1$, $i(B\cap X)\leq 2|B\cap X|-1$ and $|A\cap X|+|B\cap X|\leq |X|+1$ so
\[ 2|X|\geq 2(|A\cap X|+|B\cap X|)-2\geq i(A\cap X)+i(B\cap X)\geq 2|X|. \]
Hence, we have equality throughout and in particular $i(A\cap X)=2|A\cap X|-1$, $i(B\cap X)=2|B\cap X|-1$ and $(A\cap X)\cap(B\cap X)=\{x
\}$. Now $i((A\cap X)\cup\{a_A,b_A,c_A,d_A\})=2|A\cap X|-1+9=2(|A\cap X|+4)$, so, as $G_A$ is a $(2,1)$-circuit, we have that $(A\cap X)=A$. Similarly $(B\cap X)=B$. Hence
$X=V$, a contradiction.

Conversely, suppose $G$ is a $(2,1)$-circuit. Then for any $X\subseteq A$ we have $i(X)\leq 2|X|-1$. Let $Y\subseteq \{a_i,b_i,c_i,d_i,x\}$, then $i(Y)\leq 2|Y|-1$. As $x$ is a cut-vertex, $i(X\cup Y)\leq 2|X\cup Y|$, since $|X\cup Y|\geq |X|+|Y|-1$. Note equality holds if and only if $X=A$ and $Y=\{a_i,b_i,c_i,d_i,x\}$.
\end{proof}

\begin{lem}\label{lem:sum2a}
Let $G=(V,E)$ be a graph satisfying (2) (a) in the list above. Let $a_i,b_i,c_i,d_i$, for $i=A,B$ be eight distinct vertices not in $V$. Then $G$ is a $(2,1)$-circuit if and only if the graphs $G_i=G[i]\cup (K_4(a_i,b_i,c_i,d_i)\cup \{ xa_i, yb_i, xc_i,yd_i \}$ for $i=A,B$, are $(2,1)$-circuits. (See Figure \ref{fig:sum2a}.)
\end{lem}

\begin{figure}
\begin{center}
\begin{tikzpicture}[scale=1, fill=gray!30, vertex/.style={circle,inner sep=2,fill=black,draw}]

\filldraw [thick] plot [smooth cycle] coordinates {(0,1) (0.3,1.5) (1,1.7) (1.95,1.5) (1.85,1) (1.95,0.5) (1,0.3) (0.3,0.5)};
\node at (1,0.5) [label=north:$2|A|-2$]{};

\filldraw [thick] plot [smooth cycle] coordinates {(4,1) (3.7,1.5) (3,1.7) (2.05,1.5) (2.15,1) (2.05,0.5) (3,0.3) (3.7,0.5)};
\node at (3,0.5) [label=north:$2|B|-2$]{};

\node at (2,1.5) [vertex]{};
\node at (2,0.5) [vertex]{};

\draw[ultra thick,<->](4.5,1) -- (5,1);


\filldraw [xshift=5.5cm, thick] plot [smooth cycle] coordinates {(0,1) (0.3,1.5) (1,1.7) (1.95,1.5) (1.85,1) (1.95,0.5) (1,0.3) (0.3,0.5)};
\node at (6.5,0.5) [label=north:$2|A|-2$]{};

\draw(7.75,0.75) -- (8.25,0.75) --  (8.25,1.25) -- (7.75,1.25)-- cycle;
\draw(7.75,0.75) -- (8.25,1.25);
\draw(8.25,0.75) -- (7.75,1.25);

\draw(7.5,1.5) -- (7.75,1.25);
\draw(7.5,1.5) -- (8.25,1.25);
\draw(8.25,0.75) -- (7.5,0.5);
\draw(7.75,0.75) -- (7.5,0.5);

\node at (7.75,0.75) [vertex]{};
\node at (7.75,1.25) [vertex]{};
\node at (8.25,0.75) [vertex]{};
\node at (8.25,1.25) [vertex]{};
\node at (7.5,1.5) [vertex]{};
\node at (7.5,0.5) [vertex]{};


\filldraw [xshift=7.5cm, thick] plot [smooth cycle] coordinates {(4,1) (3.7,1.5) (3,1.7) (2.05,1.5) (2.15,1) (2.05,0.5) (3,0.3) (3.7,0.5)};
\node at (10.5,0.5) [label=north:$2|B|-2$]{};

\draw(8.75,0.75) -- (9.25,0.75) --  (9.25,1.25) -- (8.75,1.25)-- cycle;
\draw(8.75,0.75) -- (9.25,1.25);
\draw(9.25,0.75) -- (8.75,1.25);

\draw(9.5,1.5) -- (8.75,1.25);
\draw(9.5,1.5) -- (9.25,1.25);
\draw(9.25,0.75) -- (9.5,0.5);
\draw(8.75,0.75) -- (9.5,0.5);

\node at (8.75,0.75) [vertex]{};
\node at (8.75,1.25) [vertex]{};
\node at (9.25,0.75) [vertex]{};
\node at (9.25,1.25) [vertex]{};
\node at (9.5,1.5) [vertex]{};
\node at (9.5,0.5) [vertex]{};

\end{tikzpicture}
\end{center}
\caption{Illustration of the sum move for Case (2a).}
\label{fig:sum2a}
\end{figure}
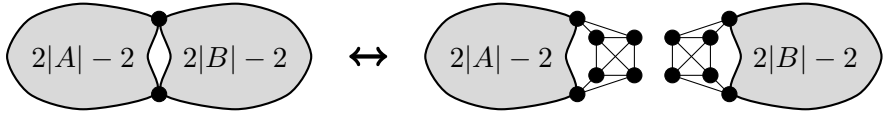

Note that the proof holds for other choices of the 4-edge-cuts in $G_A$ and $G_B$, provided two edges are incident with $x$, two are incident with $y$, simplicity is maintained and we do not choose the edge $xy$. 

\begin{proof}
Suppose $G_A$ and $G_B$ are $(2,1)$-circuits. Assume for a contradiction that $G$ is not. Then there is a proper subset $X\subsetneq V$, containing at least one of $x$ or $y$, with $i(X)=i(A\cap X)+i(B\cap X)\geq 2|X|$. First suppose that $A\cap X$ contains both $x$ and $y$. Then $i(A\cap X)\leq 2|A\cap X|-2$ with equality if and only if $X\cap A=A$; similarly $i(B\cap X)\leq 2|B\cap X|-2$; and $|A\cap X|+|B\cap X|\leq |X|+2$ so
\[ 2|X|\geq 2(|A\cap X|+|B\cap X|)-4\geq i(A\cap X)+i(B\cap X)\geq 2|X|. \]
Hence, we have equality throughout and in particular $i(A\cap X)=2|A\cap X|-2$, $i(B\cap X)=2|B\cap X|-2$. 
As noted above, this implies that $(A\cap X)=A$ and $(B\cap X)=B$. Hence $X=V$, a contradiction. 

Now suppose $A\cap X$ contains $x$ but not $y$. 
Then $|A\cap X|+|B\cap X|\leq |X|+1$.
Thus
\[ 2|X|\geq 2(|A\cap X|+|B\cap X|)-2\geq i(A\cap X)+i(B\cap X)\geq 2|X|. \]
Hence, we have equality throughout and in particular $i(A\cap X)=2|A\cap X|-1$ and $i(B\cap X)=2|B\cap X|-1$. This is a contradiction as $(A\cap X)\cup (B\cap X)$ would induce a $(2,1)$-circuit not containing $y$.

Conversely, suppose $G$ is a $(2,1)$-circuit. First note that for any subset $X\subseteq A$ containing $x$ but not $y$ and any subset $Y\subseteq \{a_A,b_A,c_A,d_A\}$ we have $i(X\cup Y)\leq 2|X\cup Y|-1$.
Now for any $X\subseteq A$ containing $x$ and $y$ we have $i(X)\leq 2|X|-2$ with equality if and only if $X=A$. Let $Y\subseteq \{a_i,b_i,c_i,d_i\}$, then $i(Y)\leq 2|Y|-2$ with equality if and only if $Y= \{a_i,b_i,c_i,d_i\}$. Thus $i(X\cup Y)\leq 2|X|-2+2|Y|-2+4= 2|X\cup Y|$ where equality holds if and only if $X=A$ and $Y=\{a_i,b_i,c_i,d_i\}$. 
\end{proof}

\begin{lem}\label{lem:invXtoT1}
Let $G$ be a 4-regular $(2,1)$-circuit which contains a cut-pair $\{x,y\}$ such that one component of $G[V\setminus\{x,y\}]$ is isomorphic to $K_4$ on the vertices $a,b,c,d$; and $\{xa,xb,yc,yd\}\in E$ where $z_1,z_2\in V\setminus\{a,b,c,d,y\}$ are neighbours of $x$. Then $x$ is admissible.
\end{lem}

\begin{proof}
The edges $z_1a,z_2b\not\in G$; so performing an inverse X-replacement on $x$ that introduces the edges $z_1a$ and $z_2b$ yields a connected 4-regular graph. Hence, by Lemma \ref{lem:4reg}, the resulting graph is a $(2,1)$-circuit.
\end{proof}

\begin{lem}\label{lem:sum2b}
Let $G=(V,E)$ be a graph satisfying (2) (b) in the list above. Let $a_B,b_B,c_B,\break d_B,e_B$ be five distinct vertices not in $V$. Then $G$ is a $(2,1)$-circuit if and only if the graphs $G_A=G[A] \cup xy$ and $G_B=G[B]\cup (K_5(a_B,b_B,c_B,d_B,e_B)-f) \cup \{ xa_B,xb_B,yc_B,yd_B \}$ are $(2,1)$-circuits. (See Figure \ref{fig:sum2b}.)
\end{lem}

\begin{figure}
\begin{center}
\begin{tikzpicture}[scale=1, fill=gray!30, vertex/.style={circle,inner sep=2,fill=black,draw}]

\filldraw [thick] plot [smooth cycle] coordinates {(0,1) (0.3,1.5) (1,1.7) (1.95,1.5) (1.85,1) (1.95,0.5) (1,0.3) (0.3,0.5)};
\node at (1,0.5) [label=north:$2|A|-1$]{};

\filldraw [thick] plot [smooth cycle] coordinates {(4,1) (3.7,1.5) (3,1.7) (2.05,1.5) (2.15,1) (2.05,0.5) (3,0.3) (3.7,0.5)};
\node at (3,0.5) [label=north:$2|B|-3$]{};

\node at (2,1.5) [vertex]{};
\node at (2,0.5) [vertex]{};

\draw[ultra thick,<->](4.5,1) -- (5,1);


\filldraw [xshift=5.5cm, thick] plot [smooth cycle] coordinates {(0,1) (0.3,1.5) (1,1.7) (1.95,1.5) (1.85,1) (1.95,0.5) (1,0.3) (0.3,0.5)};
\node at (6.5,0.5) [label=north:$2|A|-1$]{};

\draw(7.5,0.5) --(7.5,1.5);

\node at (7.5,1.5) [vertex]{};
\node at (7.5,0.5) [vertex]{};


\filldraw [xshift=7.5cm, thick] plot [smooth cycle] coordinates {(4,1) (3.7,1.5) (3,1.7) (2.05,1.5) (2.15,1) (2.05,0.5) (3,0.3) (3.7,0.5)};
\node at (10.5,0.5) [label=north:$2|B|-3$]{};

\draw(9,0.75) -- (8.5,1.5) --  (8.5,0.5) -- (9,1.25)-- (8,1) --cycle;
\draw(9,0.75) -- (8.5,0.5) -- (8,1) -- (8.5,1.5) -- (9,1.25);

\draw(9,0.75) -- (9.5,0.5) -- (8.5,0.5);
\draw(9,1.25) -- (9.5,1.5) -- (8.5,1.5);

\node at (9,0.75) [vertex]{};
\node at (8.5,0.5) [vertex]{};
\node at (8,1) [vertex]{};
\node at (8.5,1.5) [vertex]{};
\node at (9,1.25) [vertex]{};

\node at (9.5,1.5) [vertex]{};
\node at (9.5,0.5) [vertex]{};

\end{tikzpicture}
\end{center}
\caption{Illustration of the sum move for Case (2b).}
\label{fig:sum2b}
\end{figure}
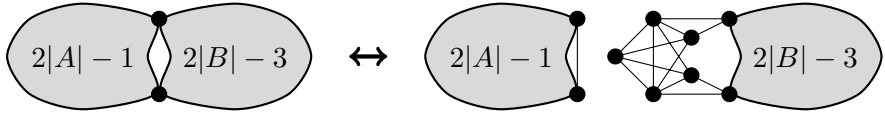

Note that the proof holds for other choices of the 4-edge-cuts in $G_A$ and $G_B$, provided two edges are incident with $x$, two are incident with $y$, simplicity is maintained and we do not choose the edge $xy$. Moreover the choice of the missing edge in $K_5-f$ is arbitrary.

\begin{proof}
Suppose $G_A$ and $G_B$ are $(2,1)$-circuits. Assume for a contradiction that $G$ is not. Then there is a proper subset $X\subsetneq V$, containing at least one of $x$ or $y$, with $i(X)=i(A\cap X)+i(B\cap X)\geq 2|X|$. First suppose that $A\cap X$ contains both $x$ and $y$. Then $i(A\cap X)\leq 2|A\cap X|-1$ with equality if and only if $X\cap A=A$. Similarly $i(B\cap X)\leq 2|B\cap X|-3$ and $|A\cap X|+|B\cap X|\leq |X|+2$ so
\[ 2|X|\geq 2(|A\cap X|+|B\cap X|)-4\geq i(A\cap X)+i(B\cap X)\geq 2|X|. \]
Hence, we have equality throughout. Hence, $(A\cap X)=A$. 
Moreover, $i(B\cap X)=2|B\cap X|-3$. Now $i((B\cap X)\cup\{a_B,b_B,c_B,d_B,e_B\})=2|B\cap X|-3+13=2(|B\cap X|+5)$, so, as $G_B$ is a $(2,1)$-circuit, we have that $(B\cap X)=B$. Hence $X=V$, a contradiction.

Now suppose $A\cap X$ contains $x$ but not $y$. 
Then $|A\cap X|+|B\cap X|\leq |X|+1$.
Thus
\[ 2|X|\geq 2(|A\cap X|+|B\cap X|)-2\geq i(A\cap X)+i(B\cap X)\geq 2|X|. \]
Hence, we have equality throughout and in particular $i(A\cap X)=2|A\cap X|-1$ and $i(B\cap X)=2|B\cap X|-1$. This is a contradiction as $(A\cap X)\cup (B\cap X)$ would induce a $(2,1)$-circuit not containing $y$.

Conversely, suppose $G$ is a $(2,1)$-circuit. Clearly $G_A$ is a $(2,1)$-circuit.
Note that for any subset $X\subseteq B$ containing $x$ but not $y$ we have $i(X)\leq 2|X|-2$ and so for any subset $Y\subseteq \{a_B,b_B,c_B,d_B,e_B\}$ we have $i(X\cup Y)\leq 2|X\cup Y|-1$. 

Now for any $X\subseteq B$ containing $x$ and $y$ we have $i(X)\leq 2|X|-3$ with equality if and only if $X=B$. Let $Y\subseteq \{a_B,b_B,c_B,d_B,e_B\}$, then $i(Y)\leq 2|Y|-1$ with equality if and only if $Y=\{a_B,b_B,c_B,d_B,e_B\}$. Thus $i(X\cup Y)\leq 2|X|-3+2|Y|-1+4= 2|X\cup Y|$ where equality holds if and only if $X=B$ and $Y= \{a_B,b_B,c_B,d_B,e_B\}$. 
\end{proof}

We emphasise that in the following lemma that the edge $xy$ is included in the count for $i_G(A)$ and $i_G(B)$.

\begin{lem}\label{lem:sum3}
Let $G=(V,E)$ be a graph satisfying (3) in the list above. Let $a_A,b_A,c_A,d_A,\break a_B,b_B,c_B,d_B,e_B$ be nine distinct vertices not in $V$. Then $G$ is a $(2,1)$-circuit if and only if the graphs $G_A=(G[A]-xy)\cup K_4(a_A,b_A,c_A,d_A)\cup \{xa_A,xb_A,yc_A,yd_A\}$ and $G_B=(G[B]-xy)\cup (K_5(a_B,b_B,c_B,d_B,e_B)-f)\cup \{xa_B,xb_B,yc_B,yd_B\}$ are $(2,1)$-circuits. (See Figure \ref{fig:sum3}.)
\end{lem}

\begin{figure}
\begin{center}
\begin{tikzpicture}[scale=1, fill=gray!30, vertex/.style={circle,inner sep=2,fill=black,draw}]

\filldraw [thick] plot [smooth cycle] coordinates {(0,1) (0.3,1.5) (1,1.7) (1.95,1.5) (1.85,1) (1.95,0.5) (1,0.3) (0.3,0.5)};
\node at (1,0.5) [label=north:$2|A|-1$]{};

\filldraw [thick] plot [smooth cycle] coordinates {(4,1) (3.7,1.5) (3,1.7) (2.05,1.5) (2.15,1) (2.05,0.5) (3,0.3) (3.7,0.5)};
\node at (3,0.5) [label=north:$2|B|-2$]{};

\draw (2,1.5) -- (2,0.5);

\node at (2,1.5) [vertex]{};
\node at (2,0.5) [vertex]{};

\draw[ultra thick,<->](4.5,1) -- (5,1);


\filldraw [xshift=5.5cm, thick] plot [smooth cycle] coordinates {(0,1) (0.3,1.5) (1,1.7) (1.95,1.5) (1.85,1) (1.95,0.5) (1,0.3) (0.3,0.5)};
\node at (6.5,0.5) [label=north:$2|A|-2$]{};

\node at (7.5,1.5) [vertex]{};
\node at (7.5,0.5) [vertex]{};

\draw(7.75,0.75) -- (8.25,0.75) --  (8.25,1.25) -- (7.75,1.25)-- cycle;
\draw(7.75,0.75) -- (8.25,1.25);
\draw(8.25,0.75) -- (7.75,1.25);

\draw(7.5,1.5) -- (7.75,1.25);
\draw(7.5,1.5) -- (8.25,1.25);
\draw(8.25,0.75) -- (7.5,0.5);
\draw(7.75,0.75) -- (7.5,0.5);

\node at (7.75,0.75) [vertex]{};
\node at (7.75,1.25) [vertex]{};
\node at (8.25,0.75) [vertex]{};
\node at (8.25,1.25) [vertex]{};
\node at (7.5,1.5) [vertex]{};
\node at (7.5,0.5) [vertex]{};


\filldraw [xshift=8.25cm, thick] plot [smooth cycle] coordinates {(4,1) (3.7,1.5) (3,1.7) (2.05,1.5) (2.15,1) (2.05,0.5) (3,0.3) (3.7,0.5)};
\node at (11.25,0.5) [label=north:$2|B|-3$]{};

\draw[xshift=0.75cm] (9,0.75) -- (8.5,1.5) --  (8.5,0.5) -- (9,1.25)-- (8,1) --cycle;
\draw[xshift=0.75cm] (9,0.75) -- (8.5,0.5) -- (8,1) -- (8.5,1.5) -- (9,1.25);

\draw[xshift=0.75cm] (9,0.75) -- (9.5,0.5) -- (8.5,0.5);
\draw[xshift=0.75cm] (9,1.25) -- (9.5,1.5) -- (8.5,1.5);

\node at (9.75,0.75) [vertex]{};
\node at (9.25,0.5) [vertex]{};
\node at (8.75,1) [vertex]{};
\node at (9.25,1.5) [vertex]{};
\node at (9.75,1.25) [vertex]{};

\node at (10.25,1.5) [vertex]{};
\node at (10.25,0.5) [vertex]{};

\end{tikzpicture}
\end{center}
\caption{Illustration of the sum move for Case (3).}
\label{fig:sum3}
\end{figure}
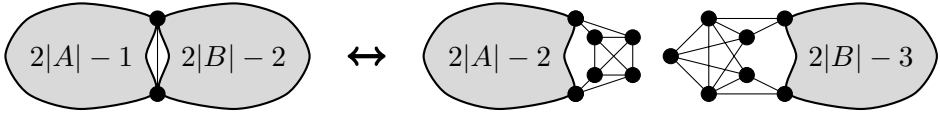

Note again that the proof holds for other choices of the 4-edge-cuts in $G_A$ and $G_B$, provided two edges are incident with $x$, two are incident with $y$, simplicity is maintained and we do not choose the edge $xy$. Moreover the choice of the missing edge in $K_5-f$ is arbitrary.

\begin{proof}
Suppose $G_A$ and $G_B$ are $(2,1)$-circuits. Assume, for a contradiction, that $G$ is not. Then there is a proper subset $X\subsetneq V$, containing at least one of $x$ or $y$, with either:
\begin{itemize}
\item $i_G(X)=i_{G_A}(A\cap X)+i_{G_B}(B\cap X)-1\geq 2|X|$, if $x,y\in X$; or
\item $i_G(X)=i_{G_A}(A\cap X)+i_{G_B}(B\cap X)\geq 2|X|$, if $x\in X$ and $y\not\in X$. 
\end{itemize}

First suppose that $x,y\in X$. Then $i_G(A\cap X)\leq 2|A\cap X|-1$ with equality if and only if $X\cap A=A$. Similarly, $i_G(B\cap X)\leq 2|B\cap X|-2$ and $|A\cap X|+|B\cap X|= |X|+2$ so
\[ 2|X|= 2(|A\cap X|+|B\cap X|)-4\geq i_G(A\cap X)+i_G(B\cap X)-1\geq 2|X|. \]
Hence, we have equality throughout. So, $(A\cap X)=A$.
Also $i_G(B\cap X)=2|B\cap X|-2$. Now $i_{G_B}((B\cap X)\cup\{a_B,b_B,c_B,d_B,e_B\})=2|B\cap X|-3+13=2(|B\cap X|+5)$, so, as $G_B$ is a $(2,1)$-circuit, we have that $(B\cap X)=B$. Hence $X=V$, a contradiction.

Now, suppose that $x\in X$ but $y\not\in X$.
Then $|A\cap X|+|B\cap X|\leq |X|+1$.
Thus
\[ 2|X|\geq 2(|A\cap X|+|B\cap X|)-2\geq i_{G_A}(A\cap X)+i_{G_B}(B\cap X)\geq 2|X|. \]
Hence, we have equality throughout and in particular $i_G(B\cap X)=i_{G_B}(B\cap X)=2|B\cap X|-1$. Then $(B\cap X)\cup\{a_B,b_B,c_B,d_B,e_B\}$ would induce a $(2,1)$-circuit in $G_B$ not containing $y$, a contradiction.

Conversely, suppose $G$ is a $(2,1)$-circuit. 
For any subset $X\subseteq A$ containing $x$ but not $y$ and any subset $Y\subseteq \{a_A,b_A,c_A,d_A\}$ we have $i_{G_A}(X\cup Y)\leq 2|X\cup Y|-1$.
Now, for any $X\subseteq A$ containing $x$ and $y$ we have $i_G(X)\leq 2|X|-1$ with equality if and only if $X=A$. Let $Y\subseteq \{a_A,b_A,c_A,d_A\}$, then $i_{G_A}(Y)\leq 2|Y|-2$ with equality if and only if $Y=\{a_A,b_A,c_A,d_A\}$. Thus $i_{G_A}(X\cup Y)\leq 2|X|-2+2|Y|-2+4= 2|X\cup Y|$ where equality holds if and only if $X=A$ and $Y= K_4(a_A,b_A,c_A,d_A)$. 

For any subset $X\subseteq B$ containing $x$ but not $y$ we have that $i_G(X)\leq 2|X|-2$ and for any subset $Y\subseteq \{a_B,b_B,c_B,d_B,e_B\}$ we have $i_{G_A}(X\cup Y)\leq 2|X\cup Y|-1$, otherwise we contradict $G$ being a $(2,1)$-circuit.

Now for any $X\subseteq B$ containing $x$ and $y$ we have $i_G(X)\leq 2|X|-2$ with equality if and only if $X=B$. Let $Y\subseteq \{a_B,b_B,c_B,d_B,e_B\}$, then $i_{G_A}(Y)\leq 2|Y|-1$ with equality if and only if $Y= \{a_B,b_B,c_B,d_B,e_B\}$. Thus $i_{G_A}(X\cup Y)\leq 2|X|-2-1+2|Y|-1+4= 2|X\cup Y|$ where equality holds if and only if $X=B$ and $Y= \{a_B,b_B,c_B,d_B,e_B\}$. 
\end{proof}

\begin{lem}\label{lem:sum4}
Let $G=(V,E)$ be a graph satisfying (4) in the list above. Let $a_A,a_B,b_B,c_B,\break d_B,e_B$ be six distinct vertices not in $V$. Then $G$ is a $(2,1)$-circuit if and only if the graphs $G_A=G[A]\cup a_A \cup \{x_1a_A,x_2a_A,x_3a_A \}$ and $G_B=G[B]\cup K_5(a_B,b_B,c_B,d_B,e_B)-f \cup \{y_1r_1,y_2r_2,y_3r_3\}$, for $r_i\in \{a_B,b_B,c_B,d_B,e_B\}$, are $(2,1)$-circuits. (See Figure \ref{fig:sum4}.)
\end{lem}

Note that we may choose all the $r_i$'s to be equal giving a stronger final characterisation as then we only have one move for this case.

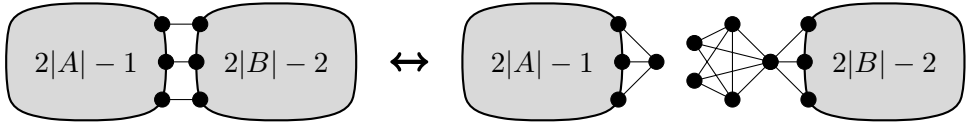
\begin{figure}
\begin{center}
\begin{tikzpicture}[scale=1, fill=gray!30, vertex/.style={circle,inner sep=2,fill=black,draw}]

\filldraw [thick] plot [smooth cycle] coordinates {(0,0.5) (0,1.5) (0.5,1.75) (1.5,1.75) (2,1.5) (2,0.5) (1.5,0.25) (0.5,0.25)};
\node at (1,0.5) [label=north:$2|A|-1$]{};

\filldraw [xshift=2.5cm, thick] plot [smooth cycle] coordinates {(0,0.5) (0,1.5) (0.5,1.75) (1.5,1.75) (2,1.5) (2,0.5) (1.5,0.25) (0.5,0.25)};
\node at (3.5,0.5) [label=north:$2|B|-2$]{};

\draw (2,1.5) -- (2.5,1.5);
\draw (2.05,1) -- (2.45,1);
\draw (2,0.5) -- (2.5,0.5);

\node at (2,1.5) [vertex]{};
\node at (2.05,1) [vertex]{};
\node at (2,0.5) [vertex]{};

\node at (2.5,1.5) [vertex]{};
\node at (2.45,1) [vertex]{};
\node at (2.5,0.5) [vertex]{};

\draw[ultra thick,<->](5,1) -- (5.5,1);


\filldraw [xshift=6cm, thick] plot [smooth cycle] coordinates {(0,0.5) (0,1.5) (0.5,1.75) (1.5,1.75) (2,1.5) (2,0.5) (1.5,0.25) (0.5,0.25)};
\node at (7,0.5) [label=north:$2|A|-1$]{};

\draw (8,1.5) -- (8.5,1);
\draw (8.05,1) -- (8.5,1);
\draw (8,0.5) -- (8.5,1);

\node at (8,1.5) [vertex]{};
\node at (8.05,1) [vertex]{};
\node at (8,0.5) [vertex]{};

\node at (8.5,1) [vertex]{};


\filldraw [xshift=10.5cm, thick] plot [smooth cycle] coordinates {(0,0.5) (0,1.5) (0.5,1.75) (1.5,1.75) (2,1.5) (2,0.5) (1.5,0.25) (0.5,0.25)};
\node at (11.5,0.5) [label=north:$2|B|-2$]{};

\draw (10.5,1.5) -- (10,1);
\draw (10.45,1) -- (10,1);
\draw (10.5,0.5) -- (10,1);

\node at (10.5,1.5) [vertex]{};
\node at (10.45,1) [vertex]{};
\node at (10.5,0.5) [vertex]{};

\node at (10,1) [vertex]{};

\draw (10,1) -- (9,1.25) -- (9.5,0.5) -- (9.5,1.5) -- (9,0.75) -- cycle;

\draw (9,1.25) -- (9.5,1.5) -- (10,1) -- (9.5,0.5) -- (9,0.75);

\node at (10,1) [vertex]{};
\node at (9,1.25) [vertex]{};
\node at (9.5,0.5) [vertex]{};
\node at (9.5,1.5) [vertex]{};
\node at (9,0.75) [vertex]{};

\end{tikzpicture}
\end{center}
\caption{Illustration of the sum move for Case (4).}
\label{fig:sum4}
\end{figure}

\begin{proof}
Suppose $G_A$ and $G_B$ are $(2,1)$-circuits. Let $X\subseteq A$ and $Y\subseteq B$. First suppose $X$ contains $x_1,x_2,x_3$ and $Y$ contains $y_1,y_2,y_3$, then $i(Y)\leq 2|Y|-2$. As $i(X)\leq 2|X|-1$, $i(X\cup Y) \leq 2|X\cup Y|-3+3=2|V|$, with equality if and only if $X\cup Y=V$. 
Next suppose $X$ contains $x_1$ and $x_2$ but not $x_3$ and $Y$ contains $y_1$ and $y_2$ but not $y_3$. Then $i(X)\leq 2|X|-1$ and, as $i_{G_B}=\{a_B,b_B,c_B,d_B,e_B\}=9$, $i(Y)\leq 2|Y|-2$. Thus $i(X\cup Y) \leq 2|X\cup Y|-3+2$.
Now suppose $X$ contains $x_1,x_2,x_3$ and $Y$ contains $y_1$ but not $y_2,y_3$. Then $i(X)\leq 2|X|-1$ and $i(Y)\leq 2|Y|-1$. Thus $i(X\cup Y) \leq 2|X\cup Y|-2+1$. The remaining cases can be proved similarly.

Conversely, suppose $G$ is a $(2,1)$-circuit. Then for any $X\subseteq A$ we have $i(X)\leq 2|X|-1$ and when  $x_1,x_2,x_3\in X$ we have equality if and only if $X=A$. It follows that $G_A$ is a $(2,1)$-circuit.

Let $Y\subseteq \{a_B,b_B,c_B,d_B,e_B\}$, then $i(Y)\leq 2|Y|-1$ with equality if and only if $Y= \{a_B,b_B,c_B,d_B,e_B\}$. Let $X\subseteq B$. If $|X\cap\{x_1,x_2,x_3\}|=0\text{ or }1$, then $i(X)\leq 2|X|-1$, thus $i(X\cup Y)\leq 2|X|-1+2|Y|-1+1= 2|X\cup Y|-1$. If $|X\cap\{x_1,x_2,x_3\}|=2$, then $i(X)\leq 2|X|-2$, otherwise $X\cup A$ would be a $(2,1)$-circuit that does not contain all of $B$,  thus $i(X\cup Y)\leq 2|X|-2+2|Y|-1+2= 2|X\cup Y|-1$. If $|X\cap\{x_1,x_2,x_3\}|=3$, then $i(X)\leq 2|X|-2$, with equality if and only if $X=B$, thus $i(X\cup Y)\leq 2|X|-2+2|Y|-1+3= 2|X\cup Y|$, with equality if and only if $X=B$ and $Y= \{a_B,b_B,c_B,d_B,e_B\}$.
\end{proof}

Recall that, e.g. in Theorem \ref{thm:Hen2admissible}, we only want to deal with non-trivial 4-edge cutsets when there are no proper critical sets. We utilise this fact in the next lemma.

\begin{lem}\label{lem:sum5a}
Let $G=(V,E)$ be a graph satisfying (5) in the list above. Let $a_A, b_A,c_A,d_A,\break a_B,b_B,c_B,d_B$ be eight distinct vertices not in $V$ and let $\E$ denote a set of four edges of the form $xy$ where $x \in A, y\in B$. 
For $I\in\{A,B\}$, if there are exactly:
\begin{itemize} 
\item 
3 distinct $x$'s in $I$, then let  $G_I=G[I]\cup K_4(a_I,b_I,c_I,d_I) \cup \F_1$, where $\F_1$ is a set of four edges such that each has exactly one end vertex in $I$, the degree of each vertex in $G_I$ is equal to their degree in $G$ and between them they have 
 two, three or four distinct end vertices in $\{a_I,b_I,c_I,d_I\}$, call the graphs in these cases type-2, type-3 and type-4 respectively and denote this set of end vertices by $F_I$;
\item four distinct $x$'s in $I$, then let  $G_I=G[I]\cup \{a_I\}\cup\{x_1a_I,x_2a_I,x_3a_I,x_4a_I\}$;
\end{itemize}
Then we have the following.
\begin{enumerate}
 \item If $G_I$ is a $(2,1)$-circuit for each $I\in\{A,B\}$, has either four distinct $x$'s in $I$, or is either  type-3 or type-4 with no proper critical sets, or type-2 with the unique proper critical set $I\cup F_I$, then $G$ is a $(2,1)$-circuit.
\item Suppose $G$ is a $(2,1)$-circuit with no proper critical sets. Then if $G_I$, for $I\in\{A,B\}$, either has four distinct $x$'s in $I$, or is type-3 or type-4, then $G_I$ is a $(2,1)$-circuit with no proper critical sets; and if $G_I$, for $I\in\{A,B\}$, is type-2, then it is a $(2,1)$-circuit with the proper critical set $I\cup F_I$.
\end{enumerate}
(See Figure \ref{fig:sum5} for an example.)
\end{lem}

\begin{figure}
\begin{center}
\begin{tikzpicture}[scale=1, fill=gray!30, vertex/.style={circle,inner sep=2,fill=black,draw}]

\filldraw [thick] plot [smooth cycle] coordinates {(0,0.5) (0,2) (0.5,2.25) (1.5,2.25) (2,2) (2,0.5) (1.5,0.25) (0.5,0.25)};
\node at (1,0.75) [label=north:$2|A|-2$]{};

\filldraw [xshift=2.5cm, thick] plot [smooth cycle] coordinates {(0,0.5) (0,2) (0.5,2.25) (1.5,2.25) (2,2) (2,0.5) (1.5,0.25) (0.5,0.25)};
\node at (3.5,0.75) [label=north:$2|B|-2$]{};

\draw (2,2) -- (2.5,2);
\draw (2.05,1.5) -- (2.45,1.5);
\draw (2.05,1) -- (2.47,0.75);
\draw (2,0.5) -- (2.47,0.75);

\node at (2,2) [vertex]{};
\node at (2.05,1.5) [vertex]{};
\node at (2.05,1) [vertex]{};
\node at (2,0.5) [vertex]{};

\node at (2.5,2) [vertex]{};
\node at (2.45,1.5) [vertex]{};
\node at (2.47,0.75) [vertex]{};

\draw[ultra thick,<->](5,1.25) -- (5.5,1.25);


\filldraw [xshift=6cm, thick] plot [smooth cycle] coordinates {(0,0.5) (0,2) (0.5,2.25) (1.5,2.25) (2,2) (2,0.5) (1.5,0.25) (0.5,0.25)};
\node at (7,0.75) [label=north:$2|A|-2$]{};

\draw (8,2) -- (8.5,1.25);
\draw (8.05,1.5) -- (8.5,1.25);
\draw (8.05,1) -- (8.5,1.25);
\draw (8,0.5) -- (8.5,1.25);

\node at (8,2) [vertex]{};
\node at (8.05,1.5) [vertex]{};
\node at (8.05,1) [vertex]{};
\node at (8,0.5) [vertex]{};
\node at (8.5,1.25) [vertex]{};


\filldraw [xshift=10cm, thick] plot [smooth cycle] coordinates {(0,0.5) (0,2) (0.5,2.25) (1.5,2.25) (2,2) (2,0.5) (1.5,0.25) (0.5,0.25)};
\node at (11,0.75) [label=north:$2|B|-2$]{};

\node at (10,2) [vertex]{};
\node at (9.95,1.5) [vertex]{};
\node at (9.97,0.75) [vertex]{};

\draw (9,1) -- (9.5,1) -- (9.5,1.5) -- (9,1.5) -- cycle;

\draw (10,2) -- (9.5,1.5);
\draw (9,1.5) -- (9.5,1) -- (9.97,0.75);

\draw[bend left=20] (9.97,0.75) edge (9,1);

\draw (9,1) -- (9.5,1.5) -- (9.95,1.5);

\node at (9,1) [vertex]{};
\node at (9.5,1) [vertex]{};

\node at (9,1.5) [vertex]{};
\node at (9.5,1.5) [vertex]{};

\end{tikzpicture}
\end{center}
\caption{Illustration of a sum move for Case (5).}
\label{fig:sum5}
\end{figure}
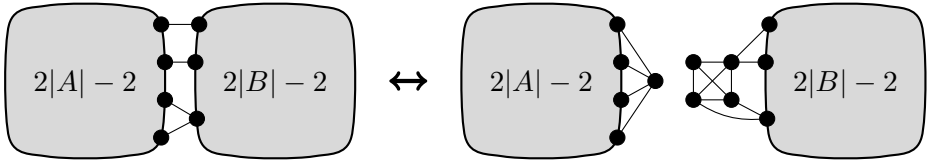

\begin{proof}
(1) Let $X\subseteq A$ and $Y\subseteq B$. 
Then $i(X)\leq 2|X|-2$ and $i(Y)\leq 2|Y|-2$ with equality if and only if $X=A$ and $Y=B$. Thus $i(X\cup Y) \leq 2|X\cup Y|-4+d(X,Y)$. We have that $i(X\cup Y)=2|X\cup Y|$ if and only if $i(X)=2|X|-2$, $i(Y)=2|Y|-2$ and $d(X,Y)=4$. This holds if and only if $X=A$ and $Y=B$ and hence $X\cup Y=V$. (Note that, if $d(X,Y)< 4$ then $i(X\cup Y)<2|X\cup Y|$.)

(2) Suppose $G$ is a $(2,1)$-circuit with no proper critical sets. 

First we consider the case where there are exactly three distinct $x$'s in $A$.  Then for any $X\subseteq A$ we have $i(X)\leq 2|X|-2$. This implies that $i(X\cup \{a_A,b_A,c_A,d_A\})\leq 2|X\cup \{a_A,b_A,c_A,d_A\}|$ with equality if and only if $X$ contains all the $x$'s in $\E$ and hence $X=A$.  Hence, $G_A$ is a $(2,1)$-circuit. Suppose that $G_A$ contains a proper critical set. Then this critical set must either be of the form $X\cup \{a_A,b_A,c_A,d_A\}$ where $X \subsetneq A$ or of the form $X\cup F_A$ where $|F_A|=2$ and $X \subseteq A$.
If the critical set is of the form $X\cup \{a_A,b_A,c_A,d_A\}$ where $X \subsetneq A$,
then $i(X\cup B)=2|X\cup B|-1$ and $|X\cup B|<|V|-1$ (since $|X|<|A|-1$), contradicting the fact that $G$ does not contain any proper critical sets. 
If $|F_A|=2$, then $A\cup F_A$ is a proper critical set.
The case where there are exactly three distinct $y$'s in $B$ is identical.

Next we consider the case where there are four distinct $x$'s in $A$. Then for any $X\subseteq A$ we have $i(X)\leq 2|X|-2$. This implies that $i(X\cup \{a_A\})\leq 2|X\cup \{a_A\}|$ with equality if and only if $X$ contains all the $x$'s in $\E$ and hence $X=A$. Suppose that $G_A$ contains a proper critical set. Then this critical set must be of the form $X\cup \{a_A\}$ where $X \subsetneq A$ and either: $X$ contains exactly three of the $x's$, in which case $i(X)=2|X|-2$; or $X$ contains all four of the $x's$, in which case $i(X)=2|X|-3$. In either case $i(X\cup B)=2|X\cup B|-1$, contradicting the fact that $G$ does not contain any proper critical sets. The case where there are four distinct $y$'s in $B$ is identical.
\end{proof}

\section{A recursive construction}
\label{sec:construction}

We need one final elementary lemma.

\begin{lem}\label{lem:forwardh2x}
Let $G$ be a $(2,1)$-circuit and let $G'$ be formed from $G$ by a 1-extension or an $X$-replacement. Then $G'$ is a $(2,1)$-circuit.
\end{lem}

Now we can prove our main result. In a similar manner to that adopted in \cite{B&J} and \cite{Nix} we refer to applications of Lemmas \ref{lem:sum1} to \ref{lem:sum5a} to combine two $(2,1)$ circuits as taking sums of connected components.
\begin{thm1.1}
A simple graph $G$ is a $(2,1)$-circuit if and only if it can be generated recursively from $H\in \G$ by applying 1-extensions and $X$-replacements sequentially within connected components and taking sums of connected components.
\end{thm1.1}

\begin{proof}
($\Leftarrow$) Suppose $G$ is recursively generated by applying 1-extensions and $X$-replacements sequentially within connected components and taking sums of connected components. Then Lemmas \ref{lem:forwardh2x}, \ref{lem:sum1}, \ref{lem:sum2a}, \ref{lem:sum2b}, \ref{lem:sum3}, \ref{lem:sum4} and \ref{lem:sum5a} together imply that $G$ is a $(2,1)$-circuit.

($\Rightarrow$) Suppose that $G \in M(2,1).$ 
\begin{itemize}
\item If $G$ is essentially 5-edge-connected and $\delta(G)=4$, then Theorem \ref{thm:4reg} implies that there exists an admissible vertex (an inverse X-replacement can be performed at this vertex).
\item If $G$ is 3-connected, essentially 5-edge-connected, contains no proper critical sets and $\delta(G)=3$, then Theorem \ref{thm:Hen2admissible} (i) (restricted to $M(2,1)$) implies that there exists an admissible node (a 1-reduction can be performed at this node).
\item If $G$ is $3$-connected, essentially 4-edge-connected, contains a proper critical set and $\delta(G)=3$, then Theorem \ref{thm:Hen2admissible} (ii) (restricted to $M(2,1)$) implies that there exists an admissible node (a 1-reduction can be performed at this node).
\end{itemize}
The cases left to consider are:
\begin{enumerate}
\item $\delta(G)=4$ and  $G$ is not essentially 5-edge-connected; or
\item $\delta(G)=3$ and either:
\begin{enumerate}
\item $G$ is 3-connected and contains non-trivial 3-edge-cutsets;
\item $G$ is 3-connected, contains no proper critical sets and is not essentially 5-edge-connected. 
\item $G$ is not 3-connected.
\end{enumerate}
\end{enumerate}
We illustrate how these cases interact during the following argument in Figure \ref{fig:flow-chart}.

\medskip

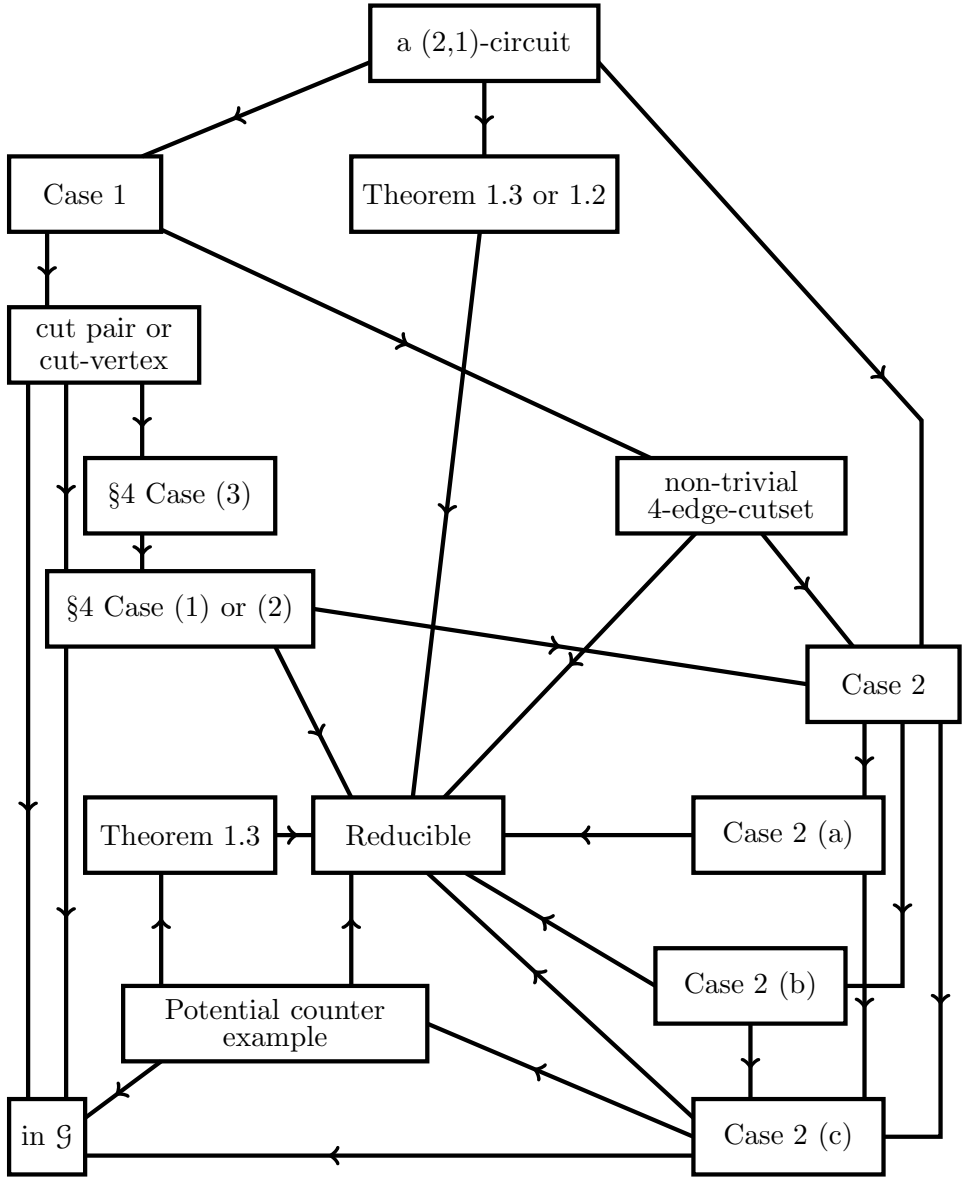
\begin{figure}
\begin{center}
\begin{tikzpicture}[scale=1, vertex/.style={circle,inner sep=2,fill=black,draw}, vertex2/.style={circle,inner sep=4,fill=black,draw}]

\draw [ultra thick, ->-=.5] (6.25,13) to (5.25,4.5);
\draw [ultra thick, ->-=.5] (1,13) to (9.5,9);
\draw [ultra thick, ->-=.5] (9.5,9) to (5.25,4.5);
\draw [ultra thick, ->-=.5] (9.5,9) to (11.5,6.5);

\draw[ultra thick] (4.75,14.5) -- (7.75,14.5) -- (7.75,15.5) -- (4.75,15.5) -- cycle;
\node at (6.25,15) {a (2,1)-circuit};


\draw[ultra thick,fill=white] (4.5,12.5) -- (8,12.5) -- (8,13.5) -- (4.5,13.5) -- cycle;
\node at (6.25,13) {Theorem \ref{thm:Hen2admissible} or \ref{thm:4reg}};


\draw[ultra thick, fill=white] (0,12.5) -- (2,12.5) -- (2,13.5) -- (0,13.5) -- cycle;
\node at (1,13) {Case 1};


\draw[ultra thick] (0,10.5) -- (2.5,10.5) -- (2.5,11.5) -- (0,11.5) -- cycle;
\node at (1.25,11.2) {cut pair or};
\node at (1.25,10.8) {cut-vertex};




\draw[ultra thick] (1,8.5) -- (3.5,8.5) -- (3.5,9.5) -- (1,9.5) -- cycle;
\node at (2.25,9) {\textsection 4 Case (3)};


\draw[ultra thick] (0.5,7) -- (4,7) -- (4,8) -- (0.5,8) -- cycle;
\node at (2.25,7.5) {\textsection 4 Case (1) or (2) };


\draw[ultra thick, fill=white] (4,4) -- (6.5,4) -- (6.5,5) -- (4,5) -- cycle;
\node at (5.25,4.5) {Reducible};




\draw[ultra thick, fill=white] (8,8.5) -- (11,8.5) -- (11,9.5) -- (8,9.5) -- cycle;
\node at (9.5,9.2) {non-trivial};
\node at (9.5,8.8) {4-edge-cutset};


\draw[ultra thick, fill=white] (10.5,6) -- (12.5,6) -- (12.5,7) -- (10.5,7) -- cycle;
\node at (11.5,6.5) {Case 2};


\draw[ultra thick] (9,4) -- (11.5,4) -- (11.5,5) -- (9,5) -- cycle;
\node at (10.25,4.5) {Case 2 (a)};


\draw[ultra thick] (8.5,2) -- (11,2) -- (11,3) -- (8.5,3) -- cycle;
\node at (9.75,2.5) {Case 2 (b)};


\draw[ultra thick] (9,0) -- (11.5,0) -- (11.5,1) -- (9,1) -- cycle;
\node at (10.25,0.5) {Case 2 (c)};


\draw[ultra thick] (1.5,1.5) -- (5.5,1.5) -- (5.5,2.5) -- (1.5,2.5) -- cycle;
\node at (3.5,2.2) {Potential counter};
\node at (3.5,1.8) {example};


\draw[ultra thick] (0,0) -- (1,0) -- (1,1) -- (0,1) -- cycle;
\node at (0.5,0.5) {in $\mathcal{G}$};


\draw[ultra thick] (1,4) -- (3.5,4) -- (3.5,5) -- (1,5) -- cycle;
\node at (2.25,4.5) {Theorem \ref{thm:Hen2admissible}};


\draw [ultra thick, ->-=.6] (6.25,14.5) to (6.25,13.5);

\draw [ultra thick, ->-=.6] (4.75,14.75) to (1.75,13.5);
\draw [ultra thick, ->-=.6] (7.75,14.75) to (12,10) to (12,7);

\draw [ultra thick, ->-=.6] (0.5,12.5) to (0.5,11.5);

\draw [ultra thick, ->-=.6] (0.25,10.5) to (0.25,1);
\draw [ultra thick, ->-=.6] (1.75,10.5) to (1.75,9.5);
\draw [ultra thick, ->-=.6] (1.75,8.5) to (1.75,8);
\draw [ultra thick, ->-=.6] (3.5,7) to (4.5,5);




\draw [ultra thick, ->-=.5] (4,7.5) to (10.5,6.5);

\draw [ultra thick, ->-=.6] (0.75,7) to (0.75,1);

\draw [ultra thick, ->-=.6] (0.75,10.5) to (0.75,8);

\draw [ultra thick, ->-=.6] (11.25,6) to (11.25,5);
\draw [ultra thick, ->-=.6] (11.75,6) to (11.75,2.5) to (11,2.5);
\draw [ultra thick, ->-=.6] (12.25,6) to (12.25,0.5) to (11.5,0.5);

\draw [ultra thick, ->-=.6] (11.25,4) to (11.25,1);
\draw [ultra thick, ->-=.6] (9.75,2) to (9.75,1);

\draw [ultra thick, ->-=.6] (9,4.5) to (6.5,4.5);
\draw [ultra thick, ->-=.6] (8.5,2.5) to (6,4);

\draw [ultra thick, ->-=.6] (9,0.75) to (5.5,4);
\draw [ultra thick, ->-=.6] (9,0.5) to (5.5,2);
\draw [ultra thick, ->-=.6] (9,0.25) to (1,0.25);

\draw [ultra thick, ->-=.6] (2,2.5) to (2,4);
\draw [ultra thick, ->-=.6] (4.5,2.5) to (4.5,4);
\draw [ultra thick, ->-=.6] (3.5,4.5) to (4,4.5);

\draw [ultra thick, ->-=.6] (2,1.5) to (1,0.75);

\end{tikzpicture}
\end{center}
\caption{The interaction between cases in the proof of Theorem \ref{thm:mainresult} .}
\label{fig:flow-chart}
\end{figure}

\medskip

\noindent\textbf{Case 1: $\mathbf{\delta(G)=4}$ and $\mathbf{G}$ is not essentially 5-edge-connected.}

First, suppose that $G$ contains a cut-vertex or a cut-pair. Then $G$ satisfies either (1), (2a), (2b) or (3) in the list at the start of Section \ref{sec:sumsec}. If $G$ satisfies (1), then we can apply Lemma \ref{lem:sum1} and reduce to two new (2,1)-circuits both with minimum degree three. If $G$ satisfies (2a), then we can apply Lemma \ref{lem:sum2a} to reduce to two new (2,1)-circuits unless one (or both) of $A$ and $B$ are isomorphic to the graph $T_1$ (the cut pair in $G$ correspond to the degree two vertices in $T_1$) shown in Figure \ref{fig:K4withwings}. (If both of $A$ and $B$ are isomorphic to $T_1$ we can apply two inverse X-replacements to achieve the graph $S_1\in\G$, see Figure \ref{fig:S234}.)

Hence, any potential counterexample arising from this case has no cut-vertices and any cut pairs, $\{x,y\}$ say, satisfy: $xy$ is not an edge and at least one of the `sides' of the cut-pair is isomorphic to $T_1$. Then, by Lemma \ref{lem:invXtoT1}, we can perform an inverse X-replacement at $x$. 
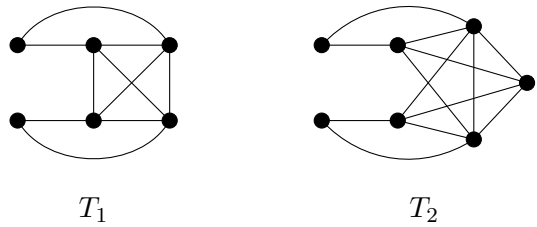
\begin{figure}
\begin{center}
\begin{tikzpicture}[scale=1, vertex/.style={circle,inner sep=2,fill=black,draw}, vertex2/.style={circle,inner sep=4,fill=black,draw}]

\coordinate (v1) at (0,0);
\coordinate (v2) at (1,0);
\coordinate (v3) at (2,0);
\coordinate (v4) at (0,1);
\coordinate (v5) at (1,1);
\coordinate (v6) at (2,1);

\node at (v1) [vertex]{};
\node at (v2) [vertex]{};
\node at (v3) [vertex]{};
\node at (v4) [vertex]{};
\node at (v5) [vertex]{};
\node at (v6) [vertex]{};

\draw (v1) -- (v3);
\draw (v4) -- (v6);
\draw (v2) -- (v5);
\draw (v2) -- (v6);
\draw (v3) -- (v5);
\draw (v3) -- (v6);

\draw[bend right=60] (v1) edge (v3);

\draw[bend left=60] (v4) edge (v6);

\node at (1,-0.75)[label=south:$T_1$]{};

\coordinate (v7) at (4,0);
\coordinate (v8) at (5,0);
\coordinate (v9) at (6,-0.25);
\coordinate (v10) at (6.7,0.5);
\coordinate (v11) at (4,1);
\coordinate (v12) at (5,1);
\coordinate (v13) at (6,1.25);

\node at (v7) [vertex]{};
\node at (v8) [vertex]{};
\node at (v9) [vertex]{};
\node at (v10) [vertex]{};
\node at (v11) [vertex]{};
\node at (v12) [vertex]{};
\node at (v13) [vertex]{};

\draw (v7) -- (v8);
\draw (v8) -- (v9);
\draw (v9) -- (v10);
\draw (v11) -- (v12);
\draw (v12) -- (v13);
\draw (v13) -- (v10);
\draw (v12) -- (v10);
\draw (v8) -- (v10);
\draw (v13) -- (v8);
\draw (v13) -- (v9);
\draw (v12) -- (v9);

\draw[bend right=40] (v7) edge (v9);

\draw[bend left=40] (v11) edge (v13);

\node at (5.35,-0.75)[label=south:$T_2$]{};

\end{tikzpicture}
\end{center}
\caption{The graphs $T_1$ and $T_2$.}
\label{fig:K4withwings}
\end{figure}

If $G$ satisfies (2b), then we can apply Lemma \ref{lem:sum2b} to reduce to two new (2,1)-circuits (one has strictly fewer vertices than $G$ and the other has minimum degree 3). If $G$ satisfies (3), then we can apply Lemma \ref{lem:sum3} to obtain two new $(2,1)$-circuits (these two circuits do not contain the edge $xy$, so we have moved to $(2,1)$-circuits satisfying (2a) or (2b)). 


Finally, suppose that $G$ is 3-connected and contains a non-trivial 4-edge-cutset, but no non-trivial 3-edge-cutset. 
\begin{claim}
\label{claim:noproper}
$G$ contains no proper critical sets.
\end{claim}
\begin{proof}[Proof of Claim]
As $\delta(G)=4$, we have that $G$ is 4-regular. Suppose that $G$ did contain a proper critical set $S$, then 
$$\sum_{v\in S}d_{G[S]}(v)=4|S|-2.$$
However, as $G$ is 4-regular this means there must be either a cut-vertex or a cut-pair in $G$, contradicting the fact that $G$ is 3-connected.
\end{proof}
Thus we can apply Lemma \ref{lem:sum5a} to achieve two new (2,1)-circuits one with strictly fewer vertices than $G$ and we can apply the lemma in such a way that the other has minimum degree three.

\noindent\textbf{Case 2a: $\mathbf{\delta(G)=3}$, $\mathbf{G}$ is 3-connected and contains non-trivial 3-edge-cutsets.}

Suppose that $G$ contains a non-trivial 3-edge-cutset. Hence $G$ contains a proper critical set. Thus we can apply Lemma \ref{lem:sum4} to form two new (2,1)-circuits. When we apply Lemma \ref{lem:sum4} we choose to do so such that the resulting circuits are strictly smaller and (by setting $r_1=r_2$) not 3-connected respectively. 

\medskip

\noindent\textbf{Case 2b: $\mathbf{\delta(G)=3}$, $\mathbf{G}$ is 3-connected, contains no proper critical sets and is not essentially 5-edge-connected.}


As $G$ does not contain a proper critical set it does not contain a non-trivial 3-edge cutset.
Hence $G$ contains a non-trivial 4-edge cutset and we can apply Lemma \ref{lem:sum5a} to form two new (2,1)-circuits. We can apply Lemma \ref{lem:sum5a} so that either the resulting circuits have strictly fewer vertices than $G$ or, by the freedom in the choice of edges when applying Lemma \ref{lem:sum5a}, we may assume that they are not 3-connected.

\medskip

\noindent\textbf{Case 2c: $\mathbf{\delta(G)=3}$ and $\mathbf{G}$ is not 3-connected.}

Suppose that there exists a cut-vertex in $G$, then we can apply Lemma \ref{lem:sum1} and reduce to two smaller (2,1)-circuits unless one (or both) of the sides of the cut are isomorphic to $K_5-e$. If both sides are isomorphic to $K_5-e$, we can, if necessary, reapply Lemma \ref{lem:sum1} until we obtain one graph,  $S_5\in\G$ (see Figure \ref{fig:base9}).

Now, suppose that there exists a cut-pair in $G$. Then $G$ satisfies (2a), (2b) or (3). If $G$ contains a cut-pair that satisfies:
\begin{itemize}
\item (2a), then we can apply Lemma \ref{lem:sum2a} to obtain two new smaller $(2,1)$-circuits, unless one side of the cut-pair is isomorphic to $T_1$ (if both sides were isomorphic to $T_1$, we would have $\delta(G)=4$);
\item  (2b), then we can apply Lemma \ref{lem:sum2b} to obtain two new $(2,1)$-circuits, either these circuits are both smaller, or one of the sides of the cut-pair is isomorphic to $T_2$, or a resulting circuit is larger than the original circuit (in the last two cases we will show that either an alternative reduction is possible or the circuit is in $\G$); or
\item (3), then we can apply Lemma \ref{lem:sum3} to obtain two new $(2,1)$-circuits, the edge $xy$ occurs in $G$ but in neither of the resulting circuits. Either the new $(2,1)$-circuits are both smaller or the resulting circuits each satisfy one of case (2a) or (2b).
\end{itemize}

Assume $G$ is a counter-example to the theorem (we will show that no such $G$ exists), then $G$ has:
\begin{itemize}
\item[(I)] every cut-vertex has one side isomorphic to $K_5-e$; and
\item[(II)] every cut-pair $\{x,y\}$ satisfies $xy\not\in E$ and has one side isomorphic to $T_1$ or $T_2$.
\end{itemize}
Associate the following multigraph $G^*$ with $G$: if (I) occurs in $G$, replace the $K_5-e$ with a loop at the cut-vertex; if (II) occurs in $G$, replace any occurrences of $T_1$ with a double edge between the vertices of the cut-pair, and any occurrences of $T_2$ with a triple edge between the vertices of the cut-pair. Note that:
\begin{itemize}
\item all the vertices in $G^\ast$ incident with a multiple edge have degree greater than three and that $G^\ast$ is 3-connected (otherwise we could have reduced $G$ using one of Lemmas \ref{lem:sum1}, \ref{lem:sum2a}, \ref{lem:sum2b} or \ref{lem:sum3});
\item if $G^\ast$ contains a loop, then the vertex incident with the loop has degree greater than three and it is not incident with any multiple edges (else both sides of the 2-cut are not isomorphic to $T_1$ or $T_2$); 
\item if $G^\ast=(V,E)$ contains a vertex, $v$ say, incident to two loops then $V=\{v\}$.
\item if there are any triple edges in $G^\ast$, then they are incident to vertices of degree greater than or equal to five, and hence, $G^\ast$ must contain vertices of degree three.
\end{itemize}
Hence $G^*\in\mathcal{M}$. See Figure \ref{fig:G_to_G*} for an illustration of the construction of $G^*$. 

\begin{figure}
\begin{center}
\begin{tikzpicture}[scale=0.8, fill=gray!30, vertex/.style={circle,inner sep=2,fill=black,draw}, vertex2/.style={circle,inner sep=0.75,fill=black,draw}]

\coordinate (v1) at (3,1);
\coordinate (v2) at (2,2);
\coordinate (v3) at (1.5,2.5);
\coordinate (v4) at (0.5,3.5);
\coordinate (v5) at (0.5,4);
\coordinate (v6) at (1.5,5);
\coordinate (v7) at (2,5.25);
\coordinate (v8) at (3.5,5.25);
\coordinate (v10) at (5.5,5.25);
\coordinate (v11) at (6,5);
\coordinate (v12) at (7,4);
\coordinate (v13) at (7,3.5);
\coordinate (v14) at (6,2.5);
\coordinate (v15) at (5.5,2);
\coordinate (v16) at (4.5,1);


\coordinate (a1) at (2.5,1);
\coordinate (a2) at (2,0.5);
\coordinate (a3) at (1.5,0.5);
\coordinate (a4) at (1.5,1);
\coordinate (a5) at (2,1.5);

\coordinate (b1) at (1,2.5);
\coordinate (b2) at (0.5,2);
\coordinate (b3) at (0,2);
\coordinate (b4) at (0,2.5);
\coordinate (b5) at (0.5,3);

\coordinate (c1) at (0.5,4.5);
\coordinate (c2) at (0,5);
\coordinate (c3) at (0,5.5);
\coordinate (c4) at (0.5,5.5);
\coordinate (c5) at (1,5);

\coordinate (d1) at (1.5,6);
\coordinate (d2) at (1.75,6.5);
\coordinate (d3) at (2.25,6.5);
\coordinate (d4) at (2.5,6);

\coordinate (e1) at (3,6);
\coordinate (e2) at (3.25,6.5);
\coordinate (e3) at (3.75,6.5);
\coordinate (e4) at (4,6);

\coordinate (f1) at (5,6);
\coordinate (f2) at (5.25,6.5);
\coordinate (f3) at (5.75,6.5);
\coordinate (f4) at (6,6);

\coordinate (g1) at (6.75,5);
\coordinate (g2) at (7,5.25);
\coordinate (g3) at (7.25,5);
\coordinate (g4) at (7,4.75);

\coordinate (h1) at (7,2.75);
\coordinate (h2) at (7.25,2.5);
\coordinate (h3) at (7,2.25);
\coordinate (h4) at (6.75,2.5);

\coordinate (i1) at (5.5,1.25);
\coordinate (i2) at (5.75,1);
\coordinate (i3) at (5.5,0.75);
\coordinate (i4) at (5.25,1);


\filldraw [ultra thick] plot [smooth cycle] coordinates {(v1) (2.7,1.7) (v2) (v3) (1.2,3.2) (v4) (v5) (1.2,4.2) (v6) (v7) (v8)  (v10) (v11) (6.3,4.3) (v12) (v13) (6.3,3.2) (v14) (v15) (4.8,1.7) (v16)};

\draw (a1) -- (a4) -- (a2) -- (a5) -- (a3) -- cycle;
\draw (b1) -- (b4) -- (b2) -- (b5) -- (b3) -- cycle;
\draw (c1) -- (c4) -- (c2) -- (c5) -- (c3) -- cycle;
\draw (d1) -- (d4) -- (d2) -- (v7) -- (d3) -- cycle;
\draw (e1) -- (e4) -- (e2) -- (v8) -- (e3) -- cycle;
\draw (f1) -- (f4) -- (f2) -- (v10) -- (f3) -- cycle;

\draw (a2) -- (v1) -- (a1) -- (a2) -- (a3) -- (a4) -- (a5) -- (v2) -- (a4);
\draw (b2) -- (v3) -- (b1) -- (b2) -- (b3) -- (b4) -- (b5) -- (v4) -- (b4);
\draw (c2) -- (v5) -- (c1) -- (c2) -- (c3) -- (c4) -- (c5) -- (v6) -- (c4);

\draw (v7) -- (d1) -- (d2);
\draw (d3) -- (d4) -- (v7);
\draw (v8) -- (e1) -- (e2);
\draw (e3) -- (e4) -- (v8);
\draw (v10) -- (f1) -- (f2);
\draw (f3) -- (f4) -- (v10);

\draw (g1) -- (g2) -- (g3) -- (g4) -- cycle;
\draw (h1) -- (h2) -- (h3) -- (h4) -- cycle;
\draw (i1) -- (i2) -- (i3) -- (i4) -- cycle;

\draw (v11) -- (g1) -- (g3) -- (v12) -- (g4) -- (g2) -- cycle;
\draw (v13) -- (h1) -- (h3) -- (v14) -- (h4) -- (h2) -- cycle;
\draw (v15) -- (i1) -- (i3) -- (v16) -- (i4) -- (i2) -- cycle;

\node at (7.7,4) [vertex2]{};
\node at (7.75,3.75) [vertex2]{};
\node at (7.7,3.5) [vertex2]{};

\node at (7.7,4) [vertex2]{};
\node at (7.75,3.75) [vertex2]{};
\node at (7.7,3.5) [vertex2]{};

\node at (6.4,1.9) [vertex2]{};
\node at (6.25,1.75) [vertex2]{};
\node at (6.1,1.6) [vertex2]{};

\node at (6.4,1.9) [vertex2]{};
\node at (6.25,1.75) [vertex2]{};
\node at (6.1,1.6) [vertex2]{};

\node at (1.1,1.9) [vertex2]{};
\node at (1.25,1.75) [vertex2]{};
\node at (1.4,1.6) [vertex2]{};

\node at (4.3,6) [vertex2]{};
\node at (4.5,6) [vertex2]{};
\node at (4.7,6) [vertex2]{};

\node at (-0.2,4) [vertex2]{};
\node at (-0.25,3.75) [vertex2]{};
\node at (-0.2,3.5) [vertex2]{};

\node at (v1) [vertex]{};
\node at (v2) [vertex]{};
\node at (v3) [vertex]{};
\node at (v4) [vertex]{};
\node at (v5) [vertex]{};
\node at (v6) [vertex]{};
\node at (v7) [vertex]{};
\node at (v8) [vertex]{};
\node at (v10) [vertex]{};
\node at (v11) [vertex]{};
\node at (v12) [vertex]{};
\node at (v13) [vertex]{};
\node at (v14) [vertex]{};
\node at (v15) [vertex]{};
\node at (v16) [vertex]{};

\node at (a1) [vertex]{};
\node at (a2) [vertex]{};
\node at (a3) [vertex]{};
\node at (a4) [vertex]{};
\node at (a5) [vertex]{};

\node at (b1) [vertex]{};
\node at (b2) [vertex]{};
\node at (b3) [vertex]{};
\node at (b4) [vertex]{};
\node at (b5) [vertex]{};

\node at (c1) [vertex]{};
\node at (c2) [vertex]{};
\node at (c3) [vertex]{};
\node at (c4) [vertex]{};
\node at (c5) [vertex]{};

\node at (d1) [vertex]{};
\node at (d2) [vertex]{};
\node at (d3) [vertex]{};
\node at (d4) [vertex]{};

\node at (e1) [vertex]{};
\node at (e2) [vertex]{};
\node at (e3) [vertex]{};
\node at (e4) [vertex]{};

\node at (f1) [vertex]{};
\node at (f2) [vertex]{};
\node at (f3) [vertex]{};
\node at (f4) [vertex]{};

\node at (g1) [vertex]{};
\node at (g2) [vertex]{};
\node at (g3) [vertex]{};
\node at (g4) [vertex]{};

\node at (h1) [vertex]{};
\node at (h2) [vertex]{};
\node at (h3) [vertex]{};
\node at (h4) [vertex]{};
 
\node at (i1) [vertex]{};
\node at (i2) [vertex]{};
\node at (i3) [vertex]{};
\node at (i4) [vertex]{};

\draw [ultra thick, ->] (8.5,3.75) -- (9,3.75);


\coordinate (u1) at (13,1);
\coordinate (u2) at (12,2);
\coordinate (u3) at (11.5,2.5);
\coordinate (u4) at (10.5,3.5);
\coordinate (u5) at (10.5,4);
\coordinate (u6) at (11.5,5);
\coordinate (u7) at (12,5.25);
\coordinate (u8) at (13.5,5.25);
\coordinate (u10) at (15.5,5.25);
\coordinate (u11) at (16,5);
\coordinate (u12) at (17,4);
\coordinate (u13) at (17,3.5);
\coordinate (u14) at (16,2.5);
\coordinate (u15) at (15.5,2);
\coordinate (u16) at (14.5,1);

\filldraw [ultra thick] plot [smooth cycle] coordinates {(u1) (12.7,1.7) (u2) (u3) (11.2,3.2) (u4) (u5) (11.2,4.2) (u6) (u7) (u8)  (u10) (u11) (16.3,4.3) (u12) (u13) (16.3,3.2) (u14) (u15) (14.8,1.7) (u16)};

\draw (u1) -- (u2);
\draw (u3) -- (u4);
\draw (u5) -- (u6);

\draw[bend left=20] (u1) edge (u2);
\draw[bend left=40] (u1) edge (u2);
\draw[bend left=20] (u3) edge (u4);
\draw[bend left=40] (u3) edge (u4);
\draw[bend left=20] (u5) edge (u6);
\draw[bend left=40] (u5) edge (u6);

\draw (u11) -- (u12);
\draw (u13) -- (u14);
\draw (u15) -- (u16);

\draw[bend left=20] (u11) edge (u12);
\draw[bend left=20] (u13) edge (u14);
\draw[bend left=20] (u15) edge (u16);

\draw plot [smooth cycle] coordinates {(u7) (11.75,6) (12,6.25) (12.25,6)};
\draw plot [smooth cycle] coordinates {(u8) (13.25,6) (13.5,6.25) (13.75,6)};
\draw plot [smooth cycle] coordinates {(u10) (15.25,6) (15.5,6.25) (15.75,6)};

\node at (17.7,4) [vertex2]{};
\node at (17.75,3.75) [vertex2]{};
\node at (17.7,3.5) [vertex2]{};

\node at (17.7,4) [vertex2]{};
\node at (17.75,3.75) [vertex2]{};
\node at (17.7,3.5) [vertex2]{};

\node at (16.4,1.9) [vertex2]{};
\node at (16.25,1.75) [vertex2]{};
\node at (16.1,1.6) [vertex2]{};

\node at (16.4,1.9) [vertex2]{};
\node at (16.25,1.75) [vertex2]{};
\node at (16.1,1.6) [vertex2]{};

\node at (11.1,1.9) [vertex2]{};
\node at (11.25,1.75) [vertex2]{};
\node at (11.4,1.6) [vertex2]{};

\node at (14.3,6) [vertex2]{};
\node at (14.5,6) [vertex2]{};
\node at (14.7,6) [vertex2]{};

\node at (9.8,4) [vertex2]{};
\node at (9.75,3.75) [vertex2]{};
\node at (9.8,3.5) [vertex2]{};

\node at (u1) [vertex]{};
\node at (u2) [vertex]{};
\node at (u3) [vertex]{};
\node at (u4) [vertex]{};
\node at (u5) [vertex]{};
\node at (u6) [vertex]{};
\node at (u7) [vertex]{};
\node at (u8) [vertex]{};
\node at (u10) [vertex]{};
\node at (u11) [vertex]{};
\node at (u12) [vertex]{};
\node at (u13) [vertex]{};
\node at (u14) [vertex]{};
\node at (u15) [vertex]{};
\node at (u16) [vertex]{};

\end{tikzpicture}
\end{center}
\caption{Illustration of the construction of the multigraph $G^*$.}
\label{fig:G_to_G*}

\end{figure}
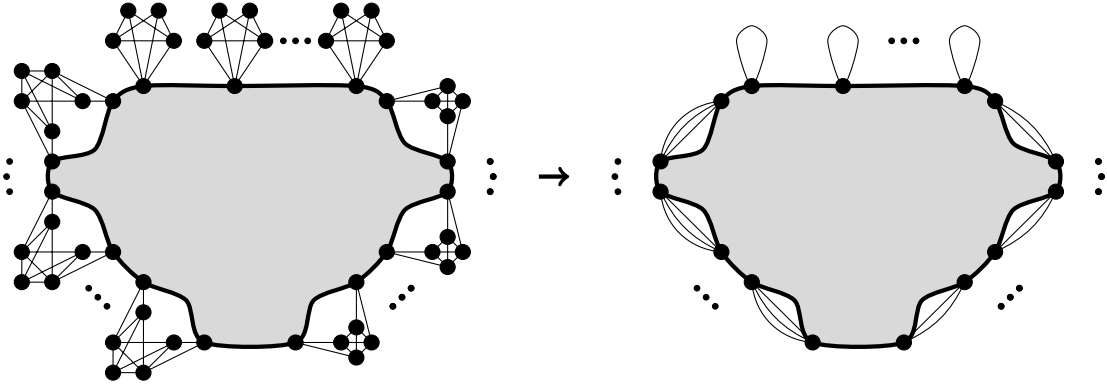

We will consider two cases, when $\delta(G^\ast)=4$ (and hence $G^\ast$ is 4-regular) and when $\delta(G^\ast)=3$.

First, suppose that $\delta(G^\ast)=4$ and $G^\ast\neq R_1$, hence $G^\ast$ is 4-regular and, as it is 3-connected, contains no proper critical sets (see Claim \ref{claim:noproper}). As $G^\ast$ is 4-regular and 3-connected it cannot contain a loop or a triple edge. If $G^\ast$ contains a double edge, then it results from a copy of $T_1$ in $G$; so, applying the same argument as in the proof of Lemma \ref{lem:invXtoT1}, we can perform an inverse X-replacement in $G$, contradicting $G$ being a counterexample.

So, suppose that $\delta(G^\ast)=3$. 
We will show that there is an admissible node $v$ in $G^*$.
Recall that $G^*$ is 3-connected and that any vertex of degree three is not incident with a loop or a multiple edge.

\begin{claim}
We may assume that the multigraph $G^*\in\mathcal{M}\setminus M(2,1)$ constructed from a counterexample $G$ does not contain any non-trivial 3-edge-cutsets.
\end{claim}

\begin{proof}[Proof of Claim]
If $G^\ast$ contains a non-trivial 3-edge-cutset, then, as $G^\ast$ is 3-connected, $G$ also contains a non-trivial 3-edge-cutset. In which case we can apply Lemma \ref{lem:sum4} to form two (2,1)-circuits. These circuits are both nonisomorphic to $G$ (contradicting our assumption that $G$ is a counterexample) unless one side of the edge-cutset is isomorphic to $K_5-e$. Moreover, we can choose the cut-edges so that they are all incident with the same vertex in the copy of $K_5-e$. At this stage we have a cut-vertex $v$ with one side isomorphic to $K_5-e$ and in our construction of $G^\ast$ this is replaced by a loop on $v$.
\end{proof}

Suppose that $G^\ast$ contains a proper critical set. As $G^\ast$ is essentially 4-edge-connected, as long as $G^\ast\not\in\G^\ast$, by Theorem \ref{thm:Hen2admissible} (ii), $G^*$ has an admissible node, and, by construction, this node is also admissible in $G$.

Finally, suppose that $G^\ast$ does not contain a proper critical set. If $G^\ast$ is essentially 5-edge-connected then we may apply Theorem \ref{thm:Hen2admissible} (i).

\begin{claim}
If the $(2,1)$-multi-circuit $G^\ast$ (containing no proper critical sets) constructed from the counter-example $G$ contains a non-trivial 4-edge-cutset, then either one of the sides of the cut is a double edge or, from $G$, we can construct a new counter-example which contains a proper critical set.
\end{claim}

\begin{proof}[Proof of Claim]
We know that $G^\ast$ is essentially 4-edge-connected. 
If $G^\ast$ contains a non-trivial 4-edge-cutset, then:
\begin{itemize}
\item if a non-trivial 4-edge-cutset is formed from four single edges, then these edges form a non-trivial 4-edge-cutset in $G$;
\item if a non-trivial 4-edge-cutset is formed from two double edges, let the endvertices of one double edge be $x_A$ and $y_A$ and the other be $x_B$ and $y_B$, these double edges resulted from two copies of $T_1$, $T_A$ and $T_B$ say, in $G$ where the degree two vertices of $T_A$ (respectively $T_B$) are $x_A$ and $y_A$ (respectively $x_B$ and $y_B$). Therefore there is a non-trivial 4-edge-cutset in $G$ consisting of the edges incident with $x_A$ and $x_B$ that are contained in $T_A$ and $T_B$.
\item if a non-trivial 4-edge-cutset is formed from two single edges and a double edge between vertices $x$ and $y$ say, then the two single edges combined with the two edges incident with $x$ in the copy of $T_1$ that results in the double edge form a non-trivial 4-edge-cutset.
\end{itemize}
So, suppose that $G^\ast$ contains a non-trivial 4-edge-cutset. As $G^\ast$ contains no proper critical sets, $G$ contains no proper critical sets. Hence we could have applied Lemma \ref{lem:sum5a} to $G$ and reduced to two $(2,1)$-circuits. These circuits are non-isomorphic to $G$ unless one or both sides of the cutset is isomorphic to $K_4$. If both sides are isomorphic to $K_4$ we either have $G=S_1\in\G$ or a graph that we can construct by summing $S_1$ with itself using Lemma \ref{lem:sum5a}. 

If only one side is isomorphic to $K_4$ a further application of Lemma \ref{lem:sum5a} guarantees that the $K_4$ contains a proper two vertex cut, $\{x,y\}$ say, this yields a proper critical set. 
\end{proof}

We may assume that, in $G^\ast$, if it contains a non-trivial 4-edge cutset then one of the sides of the cut is a double edge.

For any such 4-edge-cut in $G^\ast$ replace the two vertices forming one component with a single vertex $z$, call the resulting graph $G^{\ast\ast}$. This either results in  two vertices with four edges between them (which is in $\G^\ast$) or $z$ is a degree four vertex with either three or four neighbours (as $G^{\ast\ast}$ is 3-connected). If $z$ has four neighbours, then we have reduced to a case with one less non-trivial 4-edge-cutset. 

So assume that $z$ has three neighbours, say $u$ is the neighbour with the double edge, then the degree of $u$ is greater than three (otherwise we contradict the fact $G^\ast$ was essentially 4-edge connected). If $u$ has degree four then repeat the above over the double edge between $u$ and $z$, which is one side of a non-trivial 4-edge cutset. Hence we may assume that the degree of $u$ is at least $5$. 
In this manner we can eliminate all non-trivial 4-edge cutsets, thus we may assume that $G^{\ast\ast}$ is essentially 5-edge-connected.
If $G^{\ast\ast}\not\in\G^\ast$, then Theorem \ref{thm:Hen2admissible} (i) implies that $G^{\ast\ast}$ contains an admissible node, and, by construction, this node is also admissible in $G$.

All that remains is to consider the (2,1)-circuits $R_i$, where $0\leq i\leq 12$. First note that, as they would reduce to $R_0$, neither $R_9$ or $R_{10}$ will arise as the multigraph generated from a minimal counter example. The graphs that would yield $R_2$, $R_8$ and $R_{11}$ all contain admissible nodes. The graphs that would yield $R_1$, $R_6$, $R_7$ and $R_{12}$ all contain admissible vertices (at which an $X$-replacement can be performed). The graphs that yield $R_0$ are either isomorphic to $S_5$ or can be reduced to copies of $S_5$ through an application of Lemma \ref{lem:sum1}. 
Any graph that generates $R_3$ contains a cut pair to which Lemma \ref{lem:sum3} can be applied to yield two (2,1)-circuits, each on fewer vertices. 
Any graph that generates  $R_4$ contains a cut pair to which Lemma \ref{lem:sum3} can be applied to yield two new circuits, $G_1$ and $G_2$ say, where, without loss of generality, $G_1$ has an admissible vertex and Lemma \ref{lem:sum3} can be applied to $G_2$ to yield two (2,1)-circuits that both have admissible nodes. Finally, any graph that generates $R_5$ contains an admissible vertex.
\end{proof}

\section{Concluding remarks}

\noindent 1. We expect that our characterisation will be useful for the problem of characterising the global rigidity of realisations of graphs\footnote{A graph is realised using a map $p:V\rightarrow \mathbb{R}^3$, such a realisation is globally rigid if any other map $q$ inducing the same edge lengths is congruent to $p$. A realisation on a surface is a realisation in $\mathbb{R}^3$ in which each vertex is restricted to move only on the surface it lies on.} on surfaces of revolution (such as the cone). 
In \cite{NOP2}, rigidity of such frameworks was, generically, shown to be equivalent to the graph being $(2,1)$-tight. In \cite[Conjecture 1]{JMN} it was conjectured that the graphs which are generically globally rigid on the cone are those which are 2-connected and contain a spanning subgraph which is $(2,1)$-tight. It is not hard to prove that the graphs with $|E|=2|V|$ that are 2-connected with a spanning $(2,1)$-tight subgraph are exactly the $(2,1)$-circuits. Thus the constructive characterisation of $(2,1)$-circuits in this paper is likely to be useful in developing a combinatorial characterisation of global rigidity. See also \cite{JNstress,Nix} for analogous results for the cylinder, additional results in this direction and warnings of the additional complication for the cone.

\noindent 2.  In \cite{B&J} it was proved that all 3-connected $(2,3)$-circuits can be generated from $K_4$ by 1-extensions. (This is non-trivial since admissible nodes need not result in smaller circuits which are 3-connected even when the original circuit is 3-connected.) For $(2,1)$-circuits we now briefly comment on such extensions. It is easy to see that if $G$ is 2-connected and $v$ is an admissible node in $G$ then the resulting $(2,1)$-circuit $G'$ is 2-connected. Moreover it is easy to check that the same holds for inverse $X$-replacement and for our sum moves (clearly we do not consider the sum move for a graph with a cut-vertex). Thus we instantly have the following result. Let $\G^2$ denote all the graphs in $\G$ which are 2-connected.

\begin{cor}\label{cor:2concircuits}
Let $G=(V,E)$ be a simple graph. Then $G$ is a 2-connected $(2,1)$-circuit if and only if $G$ can be generated from some graphs in $\G^2$ by 1-extensions, $X$-replacements and sum moves.
\end{cor}

However we do not know a corresponding statement for 3-connected $(2,1)$-circuits. Indeed Figure \ref{fig:3concounterex} illustrates a $(2,1)$-circuit which is 3-connected and essentially 4-edge-connected for which each admissible 1-reduction results in a $(2,1)$-circuit which is not 3-connected.

\begin{figure}
\begin{center}
\begin{tikzpicture}[scale=1, vertex/.style={circle,inner sep=2,fill=black,draw}, vertex2/.style={circle,inner sep=4,fill=black,draw}]

\coordinate (v1) at (0,0);
\coordinate (v2) at (1,1);
\coordinate (v3) at (2,0);
\coordinate (v4) at (1,-1);
\coordinate (v5) at (-2,0);
\coordinate (v6) at (-1,1);
\coordinate (v7) at (-1,-1);
\coordinate (v8) at (0,2);
\coordinate (v9) at (-2,-1);
\coordinate (v10) at (0,3);

\node at (v1) [vertex]{};
\node at (v2) [vertex]{};
\node at (v3) [vertex]{};
\node at (v4) [vertex]{};
\node at (v5) [vertex]{};
\node at (v6) [vertex]{};
\node at (v7) [vertex]{};
\node at (v8) [vertex]{};
\node at (v9) [vertex]{};
\node at (v10) [vertex]{};

\draw (v1) -- (v2);
\draw (v1) -- (v3);
\draw (v1) -- (v4);
\draw (v1) -- (v5);
\draw (v1) -- (v6);
\draw (v1) -- (v7);
\draw (v2) -- (v3);
\draw (v2) -- (v4);
\draw (v2) -- (v8);
\draw (v3) -- (v4);
\draw (v5) -- (v6);
\draw (v5) -- (v7);
\draw (v5) -- (v8);
\draw (v5) -- (v9);
\draw (v6) -- (v7);
\draw (v6) -- (v8);
\draw (v6) -- (v9);
\draw (v7) -- (v9);
\draw (v8) -- (v10);
\draw (v5) -- (v10);
\draw (v3) -- (v10);

%

\end{tikzpicture}
\end{center}
\caption{A 3-connected $(2,1)$-circuit where no degree 4 is admissible, and every admissible node results in a $(2,1)$-circuit which is not 3-connected.}
\label{fig:3concounterex}
\end{figure}
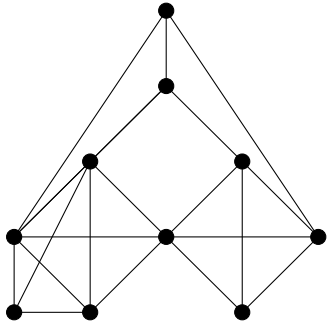

\end{document}